\documentclass[reqno]{amsart}
\usepackage[margin = 1in]{geometry}
\usepackage{amsmath, amssymb, amsthm, fancyhdr, verbatim, graphicx}
\usepackage{enumerate}
\usepackage[all]{xy}
\usepackage[dvipsnames]{xcolor}
\usepackage{mathrsfs}
\usepackage{tikz-cd}
\usepackage{framed}
\usepackage[hidelinks, pagebackref]{hyperref}
\hypersetup{colorlinks=true, linkcolor=magenta,citecolor=cyan} 
\usepackage[titletoc]{appendix}
\usepackage{bbm}
\usepackage{lipsum}
\usepackage{adjustbox}
\usepackage{commands}
\usepackage{envi}
\usepackage{stmaryrd}
\usepackage{leftindex}
\usepackage[english]{babel}
\usepackage{dynkin-diagrams}

\numberwithin{equation}{section}

\makeatletter
\def\th@remark{%
  \thm@headfont{\bfseries}%
  \normalfont 
  \thm@preskip \thm@preskip 
  \thm@postskip\thm@preskip
}
\def\imod#1{\allowbreak\mkern5mu({\operator@font mod}\,\,#1)}
\makeatother

\numberwithin{equation}{section}

\title{Diagonal cycles on Shtukas and the adjoint $L$-function}
\author{Zeyu Wang}
\address{Massachusetts Institute of Technology, Department of Mathematics, 77 Massachusetts Avenue, Cambridge, MA 02139, USA}
\email{wangzeyu@mit.edu}

\begin{document}

\begin{abstract}
We study diagonal cycles on moduli spaces of shtukas for groups of the form $H\times H$, where $H$ is a split almost simple reductive group, allowing arbitrary modification types. We relate their self-intersection numbers, with insertions of determinant line bundles, to higher derivatives of adjoint $L$-functions. This gives a general Gross--Zagier-type identity over function fields for groups of arbitrary type, and suggests a parallel conjectural picture for arithmetic intersections on Shimura varieties.
\end{abstract}

\maketitle
\tableofcontents

\section{Introduction}
\subsection{Motivation}

The Langlands program relates automorphic representations of a reductive group $G$ to Galois representations valued in the Langlands dual group $\Gc$. More generally, the relative Langlands program predicts that period integrals of automorphic forms associated to a $G$-variety $X$ should be related to special values of $L$-functions attached to the corresponding Galois representations and controlled by a dual $\Gc$-variety $\Xc$. While many instances of this phenomenon have been studied for decades, a systematic conceptual framework has only recently emerged in the celebrated work of Ben-Zvi--Sakellaridis--Venkatesh \cite{BZSV}.

The Gross--Zagier formula proved in \cite{gross1986heegner}, however, suggests the existence of an \emph{arithmetic} analogue of the relative Langlands program. It relates the N\'eron--Tate height pairing of Heegner points on elliptic curves to the central derivative of the associated $L$-function. This may be viewed as the case $G=\PGL_2$ and $X=G/H$, where $H$ is a non-split torus of $G$, with period integrals replaced by arithmetic intersection pairings of special cycles on Shimura varieties, and special values of $L$-functions replaced by special values of derivatives of $L$-functions. Various higher-dimensional generalizations of the Gross--Zagier formula have since been proposed and studied, most notably the arithmetic Gan--Gross--Prasad conjectures formulated in \cite{gan2012symplectic}. These conjectures, together with their existing generalizations, are, however, largely restricted to groups of type $A$, $B$, or $D$, and to situations in which the representation $V\in \Rep(\Gc)$ governing the Shimura variety is the standard representation. 

The celebrated work of Yun--Zhang \cite{YZ1} established a function field analogue of the Gross--Zagier formula by replacing Shimura varieties with moduli of Shtukas. In this setting, intersection numbers of special cycles on moduli of Shtukas are related to higher derivatives of $L$-functions. In the works \cite{liu2025higherperiodintegralsderivatives,wang2025specialcycleshtukascategorical}, a conceptual generalization of the Yun--Zhang formula was proposed for split groups and affine homogeneous strongly tempered spherical varieties, and proved in the case $G=\GL_n\times \GL_{n-1}$ and $X=\GL_n$. Although conceptually appealing, the framework of \cite{liu2025higherperiodintegralsderivatives,wang2025specialcycleshtukascategorical} is limited to the strongly tempered setting, which imposes a strong restriction on the representation $V\in \Rep(\Gc)$ governing the modification type of the Shtukas: namely, $V$ must occur as a direct summand of $T^*\Xc$, viewed as a $\Gc$-representation when $X$ is strongly tempered. Moreover, affine homogeneous strongly tempered spherical varieties are rather rare; see \cite{wan2025periods}.

In this article, we study the group case $G=H\times H$ and $X=H$, where $H$ is a split almost simple group. We consider the diagonal cycle in the moduli of $G$-Shtukas, allowing essentially arbitrary representations $V\in \Rep(\Gc)$ to govern the modification type. We study the intersection pairing of the diagonal cycle with itself together with the determinant line bundle, and relate it to special values of higher derivatives of adjoint $L$-functions associated to Galois representations. As far as we know, this provides the first instance of a Gross--Zagier-type formula for groups of arbitrary types.

\subsection{Main result}

Let $C$ be a smooth proper geometrically connected curve with genus $g$ over a finite field $\FF_q$ of characteristic $p$, and let $G$ be a split semisimple group over $\FF_q$. Fix a maximal torus $T\subset G$, and let $\Gc$ denote the Langlands dual group of $G$ over $k=\Qlbar$, where $\ell\neq p$. Throughout, we work with \'etale sheaves with $k$-coefficients. We assume that $p$ satisfies the hypothesis in \cite[\S0.1.9]{gaitsgory2025geometriclanglandspositivecharacteristic}, which in particular holds for sufficiently large $p$.

Fix an integer $r\geq 0$ and set $I=\{1,\dots,r\}$. For a sequence of dominant coweights $\lambda_I=(\lambda_1,\dots,\lambda_r)$ of $G$, let $\Sht_{G,\leq\lambda_I}$ denote the moduli stack of $G$-Shtukas with modifications bounded by $\lambda_I$. It is a Deligne--Mumford stack locally of finite type over $C^I$ of relative dimension
\[
d_{\lambda_I}=\sum_{i=1}^r \langle 2\rc,\lambda_i\rangle,
\]
where $2\rc$ is the sum of positive roots of $G$. We study the intersection cohomology with compact support
\[
IH_c^*(\Sht_{G,\leq\lambda_I}).
\]

Let $\cO(\Loc_{\Gc}^{\arith})$ denote the algebra of excursion operators. Given a $\Gc$-Weil local system $\s$ on $C$, let
\[
\chi_{\s}:\cO(\Loc_{\Gc}^{\arith})\to k
\]
be the associated character. By \cite{arinkin2022automorphicfunctionstracefrobenius}, the algebra $\cO(\Loc_{\Gc}^{\arith})$ acts canonically on $IH_c^*(\Sht_{G,\leq\lambda_I})$. We may therefore consider the $\s$-isotypic component
\[
IH_c^*(\Sht_{G,\leq\lambda_I})_{\s}
\subset
IH_c^*(\Sht_{G,\leq\lambda_I}),
\]
on which $\cO(\Loc_{\Gc}^{\arith})$ acts through the character $\chi_{\s}$.

We are primarily interested in the case where $\s$ is geometrically strongly irreducible.

\begin{defn}[Strongly irreducible local systems]\label{defn:stronglyirreducible}
~
\begin{itemize}
    \item A $\Gc$-local system $\overline{\s}$ on $C_{\overline{\FF_q}}$ is called \emph{strongly irreducible} if it is irreducible and its centralizer in $\Gc$ is equal to the center $Z(\Gc)$.
    
    \item A $\Gc$-Weil local system $\s$ on $C$ is called \emph{geometrically strongly irreducible} if the induced $\Gc$-local system $\overline{\s}$ on $C_{\overline{\FF_q}}$ is strongly irreducible.
\end{itemize}
\end{defn}
When $\s$ is geometrically irreducible, the $\s$-isotypic component $IH_c^*(\Sht_{G,\leq\lambda_I})_{\s}$ is a finite-dimensional vector space.

\subsubsection{Intersection pairing}

To simplify the notation in the introduction, we suppress Tate twists. There is a natural intersection pairing
\begin{equation}\label{eq:intpairingintro}
    \langle -, - \rangle_{\lambda_I} :
    IH_c^*(\Sht_{G,\leq\lambda_I})^{\otimes 2}
    \to k.
\end{equation}
See \eqref{eq:intpairing} for a more precise formulation. This pairing restricts to the isotypic components:
\begin{equation}\label{eq:intpairingisotypicintro}
    \langle -, - \rangle_{\lambda_I,\s} :
    IH_c^*(\Sht_{G,\leq\lambda_I})_{\s}
    \otimes
    IH_c^*(\Sht_{G,\leq\lambda_I})_{c_{\Gc}(\s)}
    \to k,
\end{equation}
where $c_{\Gc}:\Gc\to\Gc$ denotes the Cartan involution. Our first result establishes the nondegeneracy of this pairing for geometrically strongly irreducible local systems.

\begin{thm}[Corollary \ref{cor:intnondegenerate}]\label{thm:intnondegintro}
    For every geometrically strongly irreducible $\Gc$-Weil local system $\s$, the pairing \eqref{eq:intpairingisotypicintro} is nondegenerate.
\end{thm}

\begin{remark}
    One can remove the semisimplicity assumption on $G$ by restricting to a single connected component of $\Sht_{G,\leq\lambda_I}$. In the case $G=\GL_n$, such an analogous statement was proved in \cite[Corollary~6.13]{wang2025specialcycleshtukascategorical}.
\end{remark}

\subsubsection{Diagonal cycle and the adjoint $L$-function}

We now specialize to the case $G=H\times H$, where $H$ is a split almost simple group over $\FF_q$. Fix a maximal torus $T_H\subset H$, and set $T=T_H\times T_H$.

Let $\lambda_{H,I}=(\lambda_{H,1},\dots,\lambda_{H,r})\in X_*(T_H)^I_+$ be a sequence of dominant coweights of $H$. We associate to it the sequence
\[
\lambda_I=(\lambda_{H,I},\lambda_{H,I})\in X_*(T)^I_+.
\]
The diagonal embedding $H\hookrightarrow H\times H$ induces a morphism
\begin{equation}
    \Delta_{\Sht}:\Sht_{H,\leq\lambda_{H,I}}
    \to
    \Sht_{G,\leq\lambda_I}.
\end{equation}
It defines a Borel--Moore homology class
\begin{equation}
    \Delta_{\Sht,!}[\Sht_{H,\leq\lambda_{H,I}}]
    \in
    IH_c^{d_{\lambda_I}+2r}(\Sht_{G,\leq\lambda_I})^*,
\end{equation}
which we call the \emph{diagonal cycle} on the moduli of Shtukas.

For each $i\in I$, let
\[
\cL_{\det,i}\in \Pic(\Sht_{H,\leq\lambda_{H,I}})_{\QQ}
\]
denote the determinant line bundle associated to the $i$-th leg. We also consider the class
\begin{equation}
    \Delta_{\Sht,!}\bigl([\Sht_{H,\leq\lambda_{H,I}}]
    \cup
    \prod_{i=1}^r c_1(\cL_{\det,i})\bigr)
    \in
    IH_c^{d_{\lambda_I}}(\Sht_{G,\leq\lambda_I})^*.
\end{equation}

Let $\s_H$ be a $\Hc$-Weil local system on $C$, and define the $\Gc$-Weil local system
\[
\s=(\s_H,c_{\Hc}(\s_H)),
\]
where $c_{\Hc}$ denotes the Cartan involution of $\Hc$. We are interested in the $\s$-isotypic components of the above cycle classes:
\begin{equation}
    (\Delta_{\Sht,!}[\Sht_{H,\leq\lambda_{H,I}}])_{\s}
    \in
    IH_c^{d_{\lambda_I}+2r}(\Sht_{G,\leq\lambda_I})_{\s}^*
\end{equation}
and
\begin{equation}
    \Bigl(
    \Delta_{\Sht,!}\bigl([\Sht_{H,\leq\lambda_{H,I}}]
    \cup
    \prod_{i=1}^r c_1(\cL_{\det,i})\bigr)
    \Bigr)_{\s}
    \in
    IH_c^{d_{\lambda_I}}(\Sht_{G,\leq\lambda_I})_{\s}^*.
\end{equation}

By Theorem \ref{thm:intnondegintro}, when $\s$ is geometrically strongly irreducible, the pairing \eqref{eq:intpairingisotypicintro} induces a nondegenerate dual pairing
\begin{equation}\label{eq:intpairingisotypicintrodual}
    \langle -, -\rangle_{\lambda_I,\s}^* :
    IH_c^*(\Sht_{G,\leq\lambda_I})_{\s}^*
    \otimes
    IH_c^*(\Sht_{G,\leq\lambda_I})_{c_{\Gc}(\s)}^*
    \to k.
\end{equation}

The following theorem is the main result of this article.

\begin{thm}\label{thm:main}
Assume that $\s$ is geometrically strongly irreducible and that
$\Sht_{H,\leq\lambda_{H,I}}\neq\varnothing$. Suppose moreover that
$\s_H$ is automorphic.\footnote{That is, there exists a compactly supported
Hecke eigenfunction $f_{\s_H}$ on $\Bun_H(\FF_q)$ with eigenvalue $\s_H$.}
Then
\begin{equation}
\begin{aligned}
&
\Bigl\langle
\Bigl(
\Delta_{\Sht,!}\bigl([\Sht_{H,\leq\lambda_{H,I}}]
\cup
\prod_{i=1}^r c_1(\cL_{\det,i})\bigr)
\Bigr)_{\s},
\bigl(\Delta_{\Sht,!}[\Sht_{H,\leq\lambda_{H,I}}]\bigr)_{c_{\Gc}(\s)}
\Bigr\rangle_{\lambda_I,\s}^*
\\
&=
\frac{\epsilon_{\lambda_I}}{L(0,\frhc_{\s_H})}
\left(
\frac{1}{\log q}\frac{d}{ds}
\right)^r
\left(
q^{(g-1)\dim(\frhc) s}
L(s,\frhc_{\s_H})
\right)\Big|_{s=0},
\end{aligned}
\end{equation}
where
\begin{itemize}
    \item $L(s,\frhc_{\s_H})$ is the adjoint $L$-function associated to $\s_H$;

    \item $\dim(\frhc)$ is the dimension of the Lie algebra $\frhc$;

    \item $\epsilon_{\lambda_I}=\prod_{i=1}^r\epsilon_{\lambda_i}$, where
    $\epsilon_{\lambda_i}$ are the constants defined in \eqref{eq:epsilonformula}.
\end{itemize}
\end{thm}

\begin{remark}
The constant $\epsilon_{\lambda}$ is the Dynkin index\footnote{More precisely, it is the usual Dynkin index multiplied by the lacing number of $H$.} of the highest weight representation
$V_{\lambda_{H}}\in \Rep(\Hc)$. More precisely,
\begin{equation}\label{eq:epsilonformula}
    \epsilon_{\lambda}
    =
    \frac{\tr(XY, V_{\lambda_{H}})}{\kappa_{\min}(X,Y)}
    =
    \frac{\kappa_{\min}(\lambda_H,\lambda_H+2\rho_H)\dim V_{\lambda_{H}}}{\dim \frhc},
    \qquad X,Y\in \frhc,
\end{equation}
where $\kappa_{\min}$ is the minimal invariant bilinear form on $\frhc$,
normalized as in \eqref{eq:kappamin}, and $\rho_H$ denotes the half-sum of positive roots of $\frhc$. These constants can be viewed as simplified analogues of the eigenweights appearing in
\cite{FYZvolume,wang2026arithmeticvolumeshtukaslanglands,feng2026eigenweightsarithmetichirzebruchproportionality}.
\end{remark}

\begin{remark}
When $G=\GL_n$, every irreducible $\GL_n$-Weil local system is known to be automorphic by \cite{lafforgue2002chtoucas}. It is expected that the same holds for arbitrary reductive groups. If $\s_H$ is non-automorphic, then the left-hand side of the identity vanishes.
\end{remark}

\begin{remark}
    A generalization of Theorem \ref{thm:main} to reductive groups should be straightforward. We restrict to the almost simple case in order to have a canonical choice of determinant line bundles and to keep the formula simpler.
\end{remark}

\subsubsection{Expectation for Shimura varieties}\label{sec:shimura}

We expect that an analogue of Theorem~\ref{thm:main} with $r=1$ should hold for Shimura varieties. In particular, this predicts that the constant $\epsilon_{\mu}$ associated with the Hodge cocharacter $\mu$ of the Shimura variety appears as a multiplicity factor, a phenomenon that has not previously been observed.

For $H=\mathrm{U}(n-1,1)$ (for which $\epsilon_{\mu}=1$) and $H=\mathrm{SO}(n-2,2)$ (for which $\epsilon_{\mu}=2$), such a conjecture was formulated in \cite{CLZ}, and proved in \textit{loc.\ cit.} in the case $H=\mathrm{U}(1,1)$ under additional assumptions. Theorem~\ref{thm:main} suggests a broad generalization of this conjectural picture to arbitrary Shimura varieties, with the principal new feature being the more subtle behavior of the constant $\epsilon_{\mu}$.

We list some examples of the constants $\epsilon_{\lambda}$ for split simple groups and minuscule coweights $\lambda_H\in X_*(T_H)_+$ in Table \ref{tab:epsilon}.

\begin{table}[htbp]
    \centering
    \caption{Values of $\epsilon_{\lambda}$ for split simple groups $H$ and minuscule coweights $\lambda_H\in X_*(T_H)_+$}
    \label{tab:epsilon}
    \begin{tabular}{c c}
        \hline
        $(H,\lambda_H)$ & $\epsilon_{\lambda}$ \\
        \hline

        $(A_n,(1^i,0^{n+1-i}))$
            & $\displaystyle \binom{n-1}{i-1}$ \\

        $(B_n,~\textup{standard})$ & $2$ \\
        $(D_n,~\textup{standard})$ & $2$ \\
        $(C_n,~\textup{spin})$ & $2^{n-1}$ \\

        $(D_n,~\textup{half-spin})$ & $2^{n-3}$ \\

        $(E_6,~\textup{standard})$ & $6$ \\ 
        $(E_7,~\textup{standard})$ & $12$ \\ 
        \hline
    \end{tabular}
\end{table}

\subsection{Strategy}

We now explain the strategy of the proofs of Theorems \ref{thm:main} and \ref{thm:intnondegintro}. Although our statement is motivated by the work of \cite{CLZ}, our proof takes a fundamentally different approach from their relative trace formula method, which originates from \cite{zhang2012arithmetic}. The key ingredients of our argument are the interpretation of special cycle classes as categorical traces developed in \cite{wang2025specialcycleshtukascategorical}, together with the geometric Langlands theory in positive characteristic developed in \cite{gaitsgory2025geometriclanglandspositivecharacteristic}. These results allow us to interpret the intersection pairing as a Frobenius-twisted trace of a certain operator, which we call the \emph{intersection observable}. This construction is carried out in \S\ref{sec:diagcycle}.

The intersection observable is roughly an endomorphism of the geometric period integral. To compute it, we prove an analogue of the Atiyah--Bott formula for geometric period integrals: namely, we show that the geometric period integral is a free (odd) polynomial ring generated by tautological classes (Corollary \ref{cor:AB}), analogous to the Atiyah--Bott description of $H^*(\Bun_G)$ as a polynomial ring generated by tautological classes. Along the way, we construct an action of a Clifford algebra on the geometric period integral (Theorem \ref{thm:clifford}), which plays a key role both in the proof of the Atiyah--Bott type formula and in the computation of the intersection observable. This is the content of \S\ref{sec:geometricperiod}.

To establish the Clifford algebra action, we reduce the problem to a local computation in the Plancherel algebra using the formalism of cohomological correspondences. This is the content of \S\ref{sec:specialcohcor}.

The Clifford algebra action (Theorem \ref{thm:clifford}) also appeared in the author's previous works \cite{liu2025higherperiodintegralsderivatives,wang2025specialcycleshtukascategorical}. Compared with the strongly tempered setting treated in \textit{loc.\ cit.}, the Clifford algebra action in the present tempered setting is more intricate.

\subsection{Notations}\label{sec:notation}

We work by default with $(\infty,1)$-categories enriched over $\Vect$, the category of super vector spaces. For a category $\mathcal{C}$ and objects $x,y \in \mathcal{C}$, we denote by $\Hom(x,y)$ the mapping space from $x$ to $y$, and write $\Hom^0(x,y):=H^0(\Hom(x,y))\in \Vect^{\heartsuit}$.

Set $k=\Qlbar$. For a prestack $X$, let $\Shv(X)$ denote the category of (ind-)constructible \'etale sheaves on $X$ with coefficients in $k$. We refer to \cite[\S4.1]{liu2025higherperiodintegralsderivatives} for our conventions on this category. We write $\uk_X\in \Shv(X)$ for the constant sheaf on $X$. For $\cF\in \Shv(X)$, define
\[
H^*_{(c)}(X,\cF):=\prod_{i\in \ZZ} H^i_{(c)}(X,\cF),
\]
where $H^i_{(c)}(X,\cF):=H^i(\Gamma_{(c)}(X,\cF))$ denotes the $i$-th cohomology group (with compact support if indicated), regarded as a super vector space of parity $i\bmod 2$. We use $\langle n\rangle=\Pi^n[n](n/2)$ to denote the cohomological shift by $n$, Tate twist by $n/2$, and parity shift by $n$. We use $ \mathbb{D}^{\mathrm{Ver}}:\Shv(X)\to \Shv(X)^{\op}$ to denote the Verdier duality functor.

Let $C$ be a smooth geometrically connected projective curve over $\FF_q$. Let $D=\Spec \FF_q[\![t]\!]$ be the formal disc, and let $\Aut(D)$ denote the group of origin-preserving automorphisms of $D$. We have an isomorphism $H^*(\mathbb{B}\Aut(D))\cong k[\![\hbar]\!]$ where $\hbar=c_1(T_D)$ is the first Chern class of the tangent bundle of $D$.

For a split reductive group $G$, let $L^+G$ and $LG$ denote the jet group and loop group of $G$, respectively. We write $\Gr_G:=LG/L^+G$ for the affine Grassmannian. Define the local Hecke stack by
\[
\Hk_G^{\loc}:=(L^+G\bs \Gr_G)/\Aut(D).
\]
Its global counterpart is the Hecke stack $\Hk_G$, equipped with morphisms $\lh,\rh:\Hk_G\to \Bun_G\times C$ whose fibers are isomorphic to $\Gr_G$.

For $I=\{1,\dots,r\}$, let $\Hk_{G,I}$ denote the iterated Hecke stack with $r$ legs. It is equipped with a morphism $l_I:\Hk_{G,I}\to C^I$ recording the legs, together with morphisms $\overline{\lh}_I,\overline{\rh}_I:\Hk_{G,I}\to \Bun_G$ recording the leftmost and rightmost $G$-bundles. The moduli stack of $G$-Shtukas with $r$ legs is defined by the Cartesian diagram
\[
\begin{tikzcd}
    \Sht_{G,I} \ar[r, "f_{\Sht,I}"] \ar[d]
    & \Hk_{G,I} \ar[d, "{(\overline{\lh}_I,\Frob\circ \overline{\rh}_I)}"] \\
    \Bun_G \ar[r, "\Delta_{\Bun_G}"]
    & \Bun_G\times\Bun_G.
\end{tikzcd}
\]

For each dominant coweight $\lambda\in X_*(T)_+$, let $\Gr_{G,\leq\lambda}\subset \Gr_G$, $\Hk_{G,\leq \lambda}^{\loc}\subset \Hk_G^{\loc}$, and $\Hk_{G,\leq\lambda}\subset \Hk_G$ denote the corresponding closed Schubert substacks. Define $d_{\lambda}:=\dim \Gr_{G,\leq\lambda}$. For a sequence of dominant coweights $\lambda_I=(\lambda_1,\dots,\lambda_r)\in X_*(T)_+^I$, let $\Hk_{G,\leq\lambda_I}\subset \Hk_{G,I}$ denote the corresponding closed Schubert substack, and define $d_{\lambda_I}:=\sum_{i=1}^r d_{\lambda_i}$.

We write $\IC_{\lambda}$ (resp.\ $\IC_{\lambda_I}$) for the intersection complex on these Schubert substacks, normalized to be perverse and pure of weight zero along each fiber isomorphic to a Schubert variety, with parity given by $d_{\lambda}$ (resp.\ $d_{\lambda_I}$). We write $V_{\lambda}\in \Rep(\Gc)$ (resp.\ $V_{\lambda_I}\in \Rep(\Gc^I)$) for the irreducible representation of highest weight $\lambda$ (resp.\ $\lambda_I$). For $V\in \Rep(\Gc)$ (resp.\ $V^I\in \Rep(\Gc^I)$), we write $\Sat(V)$ (resp.\ $\Sat(V^I)$) for the corresponding sheaf on the relevant Hecke stack under the geometric Satake equivalence. For a $\Gc$-local system $\s$ on $C$, we write $V_{\lambda,\s}$ (resp.\ $V_{\lambda_I,\s}$) for the local system on $C$ (resp. $C^I$) obtained by applying $V_{\lambda}$ (resp.\ $V_{\lambda_I}$) to $\s$. We use $w_0$ to denote the longest element in the Weyl group.

When $G$ is almost simple, let $\cL_{\det}\in \Pic(\Hk_G^{\loc})_{\QQ}$ denote the determinant line bundle on the local Hecke stack, defined to be \emph{inverse} to the convention in \cite[\S4.1.1]{wang2026arithmeticvolumeshtukaslanglands}, namely, its pull-back to each connected component of $\Gr_G=LG/L^+G$ is a primitive ample generator of the Picard group. By abuse of notation, we use the same symbol $\cL_{\det}$ for its pull-back to $\Hk_G$. For each $i\in I$, let $\cL_{\det,i}\in \Pic(\Hk_{G,I})_{\QQ}$ denote the pull-back of $\cL_{\det}$ under the morphism $\Hk_{G,I}\to \Hk_G$ recording the modification at the $i$-th leg. We use the same notation for its further pull-back to $\Sht_{G,I}$.

\subsection*{Acknowledgment}
The author would like to thank his advisor, Zhiwei Yun, for constant encouragement and many invaluable discussions. He is also grateful to Ryan Chen, Weixiao Lu, Wei Zhang, and Daming Zhou for helpful discussions.

\subsection*{AI usage}
ChatGPT 5.5 and 6 Pro were used to review the article and correct typographical and grammatical errors. They also helped identify several computational errors in an early version of the article, which were subsequently corrected by the author. No other AI tools were used in the writing process. The author takes full responsibility for the content of this article.

\section{Special cohomological correspondences}\label{sec:specialcohcor}
In this section, we study special cohomological correspondences acting on the period sheaf introduced in \cite[\S4]{liu2025higherperiodintegralsderivatives}.
\subsection{The Plancherel algebra}
We begin by recalling the definition of the Plancherel algebra in the group case, following \cite[\S8]{BZSV}.
\subsubsection{Derived Satake equivalence}
Let $G$ be a split reductive group. Consider the (derived) Satake category \[(\Sat_G,*, 1_{\Sat_G}):=(\Shv(L^+G\bs LG/L^+G),*, \delta_G)\] and its variant with loop rotation \[(\Sat_{G,\hbar},*, 1_{\Sat_{G,\hbar}}):=(\Shv(L^+G\rtimes\Aut(D)\bs LG\rtimes\Aut(D)/L^+G\rtimes\Aut(D)),*,\delta_{G,\hbar})\] as monoidal categories with the convolution monoidal operation $*$. The geometric Satake equivalence proved in \cite{MV} gives canonical isomorphisms $\Sat_G^{\heartsuit}\cong \Sat_{G,\hbar}^{\heartsuit}\cong \Rep(\Gc)^{\heartsuit}$.

Consider also the renormalized version \[\Sat_G^{\ren}:=\Shv(L^+G\bs LG/L^+G)^{\ren}\] and \[\Sat_{G,\hbar}^{\ren}:=\Shv(L^+G\rtimes\Aut(D)\bs LG\rtimes\Aut(D)/L^+G\rtimes\Aut(D))^{\ren}\] where the renormalization makes constructible sheaves compact objects. One has $\Sat_G^{>-\infty}\cong \Sat_G^{\ren,>-\infty}$ and $\Sat_{G,\hbar}^{>-\infty}\cong \Sat_{G,\hbar}^{\ren,>-\infty}$ on the eventually coconnective subcategories.

The following theorem is referred to as the \emph{derived Satake equivalence}:
\begin{thm}[{\cite[Theorem\,5]{BF}}]\label{thm:derivedsatake}
    There are natural equivalences of monoidal categories \[\Sat:\QCoh((\frgc^*)^{\shear}/\Gc)\isom\Sat_G^{\ren} \] \[ \Sat:\Mod_{U_{\hbar}(\frgc^{\shear})}^{\Gc}\isom\Sat_{G,\hbar}^{\ren}.\] Here, the shearing $(\frgc^*)^{\shear}$ places $\frgc^*$ in cohomological degree $-2$, the shearing $\frgc^{\shear}$  places $\frgc$ in cohomological degree $2$, and $U_{\hbar}(\frgc^{\shear})$ is the graded enveloping algebra of $\frgc$ as considered in \cite[\S2.2]{BF}.
\end{thm}

\begin{example}\label{eg:central}
    For $V\in \Rep(\Gc)$, the corresponding object in $\Mod_{U_{\hbar}(\frgc^{\shear})}^{\Gc}$ is the $U_{\hbar}(\frgc^{\shear})$-bimodule $U_{\hbar}(\frgc^{\shear})\otimes V$ where the left action of $U_{\hbar}(\frgc^{\shear})$ is given by left multiplication on the first factor and the right action is given by $(Y\otimes v)\cdot X = YX\otimes v - \hbar Y\otimes Xv$ for $v\in V$, $Y\in U_{\hbar}(\frgc^{\shear})$, and $X\in \frgc^{\shear}$.\footnote{For the minus sign in the right action, note that it is forced by the requirement that this action define a right \(U_{\hbar}(\frgc^{\shear})\)-module structure on \(U_{\hbar}(\frgc^{\shear})\otimes V\).}
\end{example}

\subsubsection{Plancherel algebra}
From now on, take $G=H\times H$. Let $X=H$, regarded as a \emph{left} $G$-variety via the action
\[
    (h_1,h_2)\cdot h=h_1hh_2^{-1},
    \qquad (h_1,h_2)\in G,
    \quad h\in X.
\] This is called the \emph{group case} in \cite{BZSV}.
Define
\[
    \Sat_X=\Shv(L^+G\bs LX)\cong \Sat_H,
    \qquad
    \Sat_{X,\hbar}=\Shv((L^+G\rtimes\Aut(D))\bs LX)\cong \Sat_{H,\hbar}.
\]
We regard these as \emph{left} $\Sat_G$ (resp. $\Sat_{G,\hbar}$) module categories, or equivalently as $\Sat_H$ (resp. $\Sat_{H,\hbar}$)-bimodules. Let $\delta_X\in \Shv(L^+G\bs LX)$ (resp. $\delta_{X,\hbar}\in \Shv((L^+G\rtimes\Aut(D))\bs LX)$) be the basic object, namely the constant sheaf supported on $L^+G\bs L^+X$ (resp. on $(L^+G\rtimes\Aut(D))\bs L^+X$).

Define the \emph{Plancherel algebra} to be the algebra object \begin{equation}
    \PL_X:=\underline{\End}_{\Rep(\Gc)}(\delta_X)\in \Alg(\Rep(\Gc)).
\end{equation} It also has a noncommutative version
\begin{equation}
    \PL_{X,\hbar}:=\underline{\End}_{\Rep(\Gc)}(\delta_{X,\hbar})\in \Alg(\Rep(\Gc)).
\end{equation}

\subsection{Local special cohomological correspondences}
Now we study local special cohomological correspondences in the group case.
\subsubsection{Definition}
A \emph{local special cohomological correspondence} of degree $d$ is an element \[\frc_{V}^{\loc} \in \Hom^0_{\Rep(\Gc)}(V,\PL_{X,\hbar}\langle d\rangle)\cong \Hom^0_{\Sat_{X,\hbar}}(\Sat(V)*\delta_{X,\hbar},\delta_{X,\hbar}\langle d\rangle )\] for some $V\in \Rep(\Gc)$ and $d\in\ZZ$.

Given local special cohomological correspondences $\frc_{V}^{\loc}$ and $\frc_{W}^{\loc}$ of degree $d_V$ and $d_W$ respectively, one can define their composition $\frc_{V}^{\loc}\circ \frc_{W}^{\loc}\in \Hom^0_{\Rep(\Gc)}(V\otimes W,\PL_{X,\hbar}\langle d_V+d_W\rangle)$ as cohomological correspondences which is also realized by the composition \[V\otimes W\xrightarrow{\frc_V^{\loc}\otimes \frc_W^{\loc}} \PL_{X,\hbar}\langle d_V\rangle \otimes \PL_{X,\hbar}\langle d_W\rangle\xrightarrow{\rmm} \PL_{X,\hbar}\langle d_V+d_W\rangle\] where $\rmm$ is the multiplication map on $\PL_{X,\hbar}$.

Given a local special cohomological correspondence $\frc_W^{\loc}$ of degree $d$ and a morphism $f:V\to W$ in $\Rep(\Gc)$, one can define the pullback \begin{equation}\label{eq:pullbackloccor} f^*\frc_W^{\loc}\in \Hom^0_{\Rep(\Gc)}(V,\PL_{X,\hbar}\langle d\rangle)\end{equation} by the composition \[V\xrightarrow{f} W\xrightarrow{\frc_W^{\loc}} \PL_{X,\hbar}\langle d\rangle.\]

\subsubsection{Examples of local special cohomological correspondences}

For $\lambda_H\in X_*(T_H)_+$, we use $V_{\lambda_H}\in \Rep(\Hc)$ to denote the unique irreducible representation with highest weight $\lambda_H$. Write $\lambda=(\lambda_H,\lambda_H)\in X_*(T)_+$.

Define the \emph{local diagonal cohomological correspondence} (of modification type $\lambda$ and degree $0$) to be the element \begin{equation}\label{eq:diagcorloc}
    \frd_{\lambda}^{\loc}\in \Hom^0_{\Rep(\Gc)}(V_{\lambda},\PL_{X,\hbar})\cong \Hom^0_{\Sat_{H,\hbar}}(\IC_{\lambda_H}*\IC_{-w_0\lambda_H},\delta_{H,\hbar})
\end{equation}
given by the fundamental class of the Satake cycle.

Define the \emph{local Lefschetz cohomological correspondence} (of modification type $\lambda$ and degree $2$) to be the element 
\begin{equation}\label{eq:lefcorloc}
    \frd_{\det,\lambda}^{\loc}\in \Hom^0_{\Rep(\Gc)}(V_{\lambda},\PL_{X,\hbar}\langle 2\rangle) \cong \Hom^0_{\Sat_{H,\hbar}}(\IC_{\lambda_H}*\IC_{-w_0\lambda_H},\delta_{H,\hbar}\langle 2\rangle)
\end{equation} corresponding to the map \[\IC_{\lambda_H}*\IC_{-w_0\lambda_H}\xrightarrow{c_1(\cL_{\det})*\id} \IC_{\lambda_H}*\IC_{-w_0\lambda_H}\langle 2\rangle \xrightarrow{\frd_{\lambda}^{\loc}} \delta_{H,\hbar}\langle 2\rangle.\]

Define the \emph{local adjoint cohomological correspondence} (of degree 2) to be the element \begin{equation}
    \fra\frd^{\loc} 
    \in \Hom^0_{\Rep(\Gc)}(\frhc\boxtimes k, \PL_{X,\hbar}\langle 2\rangle)
    \cong \Hom^0_{\Mod_{U_{\hbar}(\frhc^{\shear})}^{\Hc}}( U_{\hbar}(\frhc^{\shear})\otimes \frhc,U_{\hbar}(\frhc^{\shear})\langle 2\rangle)\cong \Hom^0_{\Rep(\Hc)}(\frhc, U_{\hbar}(\frhc^{\shear})\langle 2\rangle) \end{equation} corresponding to the tautological map.

\subsubsection{A local commutator relation}

Consider the map \begin{equation}\act_{V_{\lambda_H}}\in \Hom^0_{\Rep(\Hc)}(\frhc\otimes V_{\lambda_H} ,V_{\lambda_H})\end{equation} given by the Lie algebra action. It induces a natural map \begin{equation}
    \act_{V_{\lambda_H}}\boxtimes \id \in \Hom^0_{\Rep(\Gc)}((\frhc\boxtimes k)\otimes V_{\lambda},V_{\lambda})
\end{equation} in the obvious way.

\begin{prop}\label{prop:localcommutator}
    The following identity holds \[[\fra\frd^{\loc},\frd_{\lambda}^{\loc}]= \hbar (\act_{V_{\lambda_H}}\boxtimes \id)^*\frd_{\lambda}^{\loc}\in\Hom^0((\frhc\boxtimes k)\otimes V_{\lambda},\PL_{X,\hbar}\langle 2\rangle).\]
\end{prop}

\begin{proof}
    Consider the diagram \begin{equation}\label{diag:localcommutatordiag1} \adjustbox{max width=\linewidth}{%
         \begin{tikzcd} U_{\hbar}(\frhc^{\shear})\otimes \frhc\otimes V_{\lambda_H}\otimes V_{-w_0\lambda_H} \ar[r, "\sim"] \ar[d, "\sim"] & U_{\hbar}(\frhc^{\shear})\otimes \frhc\otimes_{U_{\hbar}(\frhc^{\shear})}U_{\hbar}(\frhc^{\shear})\otimes V_{\lambda_H} \otimes_{U_{\hbar}(\frhc^{\shear})}U_{\hbar}(\frhc^{\shear})\otimes V_{-w_0\lambda_H} \ar[d, "\id\otimes \frd_{\lambda}^{\loc}"] \\ U_{\hbar}(\frhc^{\shear})\otimes V_{\lambda_H}\otimes_{U_{\hbar}(\frhc^{\shear})} U_{\hbar}(\frhc^{\shear})\otimes \frhc \otimes_{U_{\hbar}(\frhc^{\shear})}U_{\hbar}(\frhc^{\shear})\otimes V_{-w_0\lambda_H} \ar[d, "\id\otimes \fra\frd^{\loc}\otimes\id"] & U_{\hbar}(\frhc^{\shear})\otimes \frhc \ar[d, "\fra\frd^{\loc}"] \\ U_{\hbar}(\frhc^{\shear})\otimes V_{\lambda_H} \otimes_{U_{\hbar}(\frhc^{\shear})}U_{\hbar}(\frhc^{\shear})\otimes V_{-w_0\lambda_H}\langle 2\rangle \ar[r, "\frd_{\lambda}^{\loc}"] & U_{\hbar}(\frhc^{\shear})\langle 2\rangle .\end{tikzcd} 
        }
    \end{equation}
    Under the identification \[\Hom^0((\frhc\boxtimes k)\otimes V_{\lambda},\PL_{X,\hbar}\langle 2\rangle)\cong \Hom^0_{\Mod_{U_{\hbar}(\frhc^{\shear})}^{\Hc}}( U_{\hbar}(\frhc^{\shear})\otimes  \frhc\otimes V_{\lambda_H}\otimes V_{-w_0\lambda_H},U_{\hbar}(\frhc^{\shear})\langle 2\rangle),\] the map $\fra\frd^{\loc}\circ \frd_{\lambda}^{\loc}$ corresponds to the upper right composition, and the map $\frd_{\lambda}^{\loc}\circ \fra\frd^{\loc}$ corresponds to the lower left composition. For $1\otimes X\otimes v \otimes v^* \in U_{\hbar}(\frhc^{\shear})\otimes  \frhc\otimes V_{\lambda_H}\otimes V_{-w_0\lambda_H}$, the diagram \eqref{diag:localcommutatordiag1} applied to this element gives \begin{equation}
    \begin{tikzcd}
        1\otimes X\otimes v \otimes v^* \ar[r, mapsto]\ar[d, mapsto] 
        & 1\otimes X\otimes 1\otimes v\otimes 1\otimes v^*  \ar[d, "\frd_{\lambda}^{\loc}\otimes\id", mapsto] \\ 
        1\otimes v\otimes 1\otimes X\otimes 1\otimes v^* \ar[d, "\id\otimes \fra\frd^{\loc}\otimes\id", mapsto]  
        & \counit_{V_{\lambda_H}}(v\otimes v^*) \otimes X \ar[d, "\fra\frd^{\loc}", mapsto] \\ 
        1\otimes v\otimes X\otimes v^* \ar[r, "\frd_{\lambda}^{\loc}", mapsto] 
        & \left\{
        \begin{aligned} 
            \counit_{V_{\lambda_H}}(v\otimes v^*) X \\ 
            \counit_{V_{\lambda_H}}(v\otimes v^*) X - \counit_{V_{\lambda_H}}(Xv\otimes v^*)\hbar 
        \end{aligned}
        \right.
        .
    \end{tikzcd} 
\end{equation}
    Here, the map $\counit_{V_{\lambda_H}}:V_{\lambda_H}\otimes V_{-w_0\lambda_H}\to k$ is the natural element given by the Satake cycle. This implies that \[[\fra\frd^{\loc},\frd_{\lambda}^{\loc}](1\otimes X\otimes v \otimes v^*)= \counit_{V_{\lambda_H}}(Xv\otimes v^*) \hbar=\hbar (\act_{\lambda_H}\boxtimes \id)^*\frd_{\lambda}^{\loc}(1\otimes X\otimes v \otimes v^*)\] as desired.

\end{proof}

\subsection{Local input for the intersection observable}
In this subsection, we carry out several local computations in the Plancherel algebra that will serve as the key local input for the computation of the intersection observable in \S\ref{sec:intobs}.

Consider the natural element \begin{equation}\unit_{V_{-w_0\lambda}}\in \Hom^0_{\Rep(\Gc)}(k, V_{\lambda}\otimes V_{-w_0\lambda})\cong \Hom^0_{\Sat_G}(\IC_{\lambda}*\IC_{-w_0\lambda},\delta_G)\end{equation} corresponding to the Satake cycle. It induces a map \begin{equation}\id\otimes \unit_{V_{-w_0\lambda}}:\frhc\boxtimes k\to (\frhc\boxtimes k)\otimes V_{\lambda}\otimes V_{-w_0\lambda}\cong V_{\lambda}\otimes(\frhc\boxtimes k)\otimes V_{-w_0\lambda}.\end{equation} Consider also the map \[\act_{V_{-w_0\lambda}}\boxtimes \id: (\frhc\boxtimes k)\otimes V_{-w_0\lambda}\to V_{-w_0\lambda} \] induced by the Lie algebra action.

\subsubsection{Local input for the derivative part}

Consider the invariant bilinear form \begin{equation}\label{eq:kappamin}\kappa_{\min}:\frhc\otimes\frhc\to k\end{equation} normalized such that $\kappa_{\min}(\alpha_{\mathrm{s}},\alpha_{\mathrm{s}})=2$ for any short root $\alpha_{\mathrm{s}}\in X^*(\Tc_H)$.\footnote{That is, $\kappa_{\min}=l_H^{-1}\kappa_{\mathrm{basic}}$ where $\kappa_{\mathrm{basic}}$ is the basic invariant form on $\frhc$ and $l_H$ is the lacing number of $H$.}

\begin{prop}\label{prop:localunit}
    The following identity holds \begin{equation}\label{eq:localunit1}(\id\otimes\unit_{V_{-w_0\lambda}})^*(\frd_{\det,\lambda}^{\loc}\circ (\act_{V_{-w_0\lambda_H}}\boxtimes \id)^* \frd_{-w_0\lambda}^{\loc})=(-1)^{d_{\lambda_H}+1 }\epsilon_{\lambda}\cdot \fra\frd^{\loc}\in \Hom^0_{\Rep(\Gc)}(\frhc\boxtimes k, \PL_{X,\hbar}\langle 2\rangle)\end{equation} where the number $\epsilon_{\lambda}$ is defined in \eqref{eq:epsilonformula}.
\end{prop}

\begin{proof}
    Since the reduction modulo $\hbar$ map induces an isomorphism \[\Hom^0_{\Rep(\Gc)}(\frhc\boxtimes k, \PL_{X,\hbar}\langle 2\rangle)\cong \Hom^0_{\Rep(\Gc)}(\frhc\boxtimes k, \PL_{X}\langle 2\rangle),\] we can work with the commutative Plancherel algebra $\PL_X$ instead of the noncommutative Plancherel algebra $\PL_{X,\hbar}$. 

    For any $V\in \Rep(\Hc)$, by \cite[Proposition\,5.7]{yun2011integral}, the element \[c_1(\cL_{\det})\in \Hom^0_{\Sat_H}(\Sat(V),\Sat(V)\langle 2\rangle)\cong \Hom^0_{\QCoh((\frhc^*)^{\shear}/\Hc)}(V \otimes\cO,V\otimes \cO\langle 2\rangle)\] is given by the morphism whose fiber at $X\in \frhc^*$ is given by $v\mapsto \kappa_{\min}(X)v$ where $\kappa_{\min}:\frhc^*\isom \frhc$ is induced by \eqref{eq:kappamin}.\footnote{There is a sign error in \cite[Proposition\,5.7]{yun2011integral}, as noted and corrected in the footnote of \cite[\S5.3]{riche2017kostant}.} Under the identification \[\Hom^0_{\Rep(\Gc)}(\frhc\boxtimes k,\PL_{X}\langle 2\rangle)\cong \Hom^0_{\QCoh((\frhc^*)^{\shear}/\Hc)}(\frhc\otimes\cO,\cO\langle 2\rangle),\] the left-hand side of \eqref{eq:localunit1} modulo $\hbar$ corresponds to the composition \begin{equation}\begin{split}
         & \frhc\otimes\cO \\ 
         \xrightarrow{\id  \otimes \unit_{V_{-w_0\lambda_H}} \otimes  \unit_{V_{-w_0\lambda_H}}  } & \frhc\otimes V_{\lambda_H}\otimes V_{-w_0\lambda_H}\otimes  V_{\lambda_H}\otimes V_{-w_0\lambda_H} \otimes  \cO \\ 
         \xrightarrow{\sw\otimes\id\otimes\id\otimes\id  } & V_{\lambda_H} \otimes \frhc\otimes V_{-w_0\lambda_H}\otimes  V_{\lambda_H}\otimes V_{-w_0\lambda_H} \otimes  \cO \\ 
         \xrightarrow{\id\otimes \act_{V_{-w_0\lambda_H}} \otimes \id \otimes  \id } & V_{\lambda_H}\otimes V_{-w_0\lambda_H}\otimes  V_{\lambda_H}\otimes V_{-w_0\lambda_H} \otimes  \cO \\ 
        \xrightarrow{c_1(\cL_{\det})\otimes\counit_{V_{-w_0\lambda_H}}\otimes \id} & V_{\lambda_H}\otimes V_{-w_0\lambda_H}\otimes\cO \langle 2\rangle \\
         \xrightarrow{\counit_{V_{\lambda_H}} } & \cO\langle 2\rangle
    \end{split}\end{equation} where $\sw$ is the morphism swapping two factors.

    Using the fact that $\counit_{V_{-w_0\lambda_H}}$ and $\unit_{V_{-w_0\lambda_H}}$ define a duality datum and \[\counit_{V_{-w_0\lambda_H}}=(-1)^{d_{\lambda_H} }\counit_{V_{\lambda_H}}\circ \sw\in \Hom^0_{\Rep(\Hc)}(V_{-w_0\lambda_H}\otimes V_{\lambda_H},k),\] one checks that this composition coincides with the right-hand side of \eqref{eq:localunit1} modulo $\hbar$ as desired.
\end{proof}

\subsubsection{Local input for the scalar part}
Now we come to another local computation.

\begin{prop}\label{prop:scalarpartloc}
    We have \[\unit_{V_{-w_0\lambda}}^*(\frd_{\det,\lambda}^{\loc}\circ \frd_{-w_0\lambda}^{\loc})=(-1)^{d_{\lambda_H} +1}b_{\lambda}\hbar \in \Hom^0_{\Rep(\Gc)}(k, \PL_{X,\hbar}\langle 2\rangle)\] where \begin{equation}\label{eq:bformula}
    b_{\lambda}=\frac{\dim(\frhc)}{2}\epsilon_{\lambda}.
\end{equation}
\end{prop}

\begin{proof}
Under the identification $\Hom^0_{\Rep(\Gc)}(k, \PL_{X,\hbar}\langle 2\rangle)\cong \Hom^0_{\Sat_{H,\hbar}}(\delta_{H,\hbar},\delta_{H,\hbar}\langle 2\rangle)$, the left-hand side of the desired identity corresponds to the composition 
\begin{equation}\label{eq:bcalculation1}
    \begin{split}
        & \delta_{H,\hbar} \\ 
        \xrightarrow{\unit_{V_{-w_0\lambda_H}}*\unit_{V_{-w_0\lambda_H}}} & \IC_{\lambda_H} * \IC_{-w_0\lambda_H} * \IC_{\lambda_H} * \IC_{-w_0\lambda_H} \\
        \xrightarrow{c_1(\cL_{\det})*\counit_{V_{-w_0\lambda_H}}*\id} & \IC_{\lambda_H} * \IC_{-w_0\lambda_H} \langle 2\rangle \\
        \xrightarrow{\counit_{V_{\lambda_H}}} & \delta_{H,\hbar}\langle 2\rangle
    \end{split}
.\end{equation} Using the same identities as in the proof of Proposition \ref{prop:localunit}, one checks that this composition coincides with the categorical trace\footnote{Note that categorical trace is usually defined for symmetric monoidal categories. While $\Sat_{H,\hbar}$ is not symmetric monoidal, the left and right dual of $\IC_{\lambda_H}$ are canonically identified by using the commutative constraint for $\Sat_{H,\hbar}^{\heartsuit}\cong \Rep(\Hc)^{\heartsuit}$.} \[(-1)^{d_{\lambda_H}}\tr_{\Sat_{H,\hbar}}(c_1(\cL_{\det}),\IC_{\lambda_H})\in \Hom^0_{\Sat_{H,\hbar}}(\delta_{H,\hbar},\delta_{H,\hbar}\langle 2\rangle).\] Since $\Hom^0_{\Sat_{H,\hbar}}(\delta_{H,\hbar},\delta_{H,\hbar}\langle 2\rangle)\cong k\cdot \hbar$ is one-dimensional, we have $\tr_{\Sat_{H,\hbar}}(c_1(\cL_{\det}),\IC_{\lambda_H})=-b_{\lambda}\hbar$ for some $b_{\lambda}\in k$.

Now we compute $b_{\lambda}$ by translating \eqref{eq:bcalculation1} to the spectral side. Recall that for any $V\in \Rep(\Hc)$ the element $c_1(\cL_{\det})\in \Hom^0_{\Sat_{H,\hbar}}(U_{\hbar}(\frhc^{\shear})\otimes V,U_{\hbar}(\frhc^{\shear})\otimes V\langle 2\rangle)$ is given by \[
c_1(\cL_{\det})(a\otimes v)  = \sum_i aX_i\otimes X^iv-\frac{\hbar}{2}\sum_i a\otimes X_iX^i v
\] where $a\in U_{\hbar}(\frhc^{\shear})$, $v\in V$, and $\{X_i\}$ and $\{X^i\}$ are dual bases of $\frhc$ with respect to the invariant bilinear form $\kappa_{\min}$.\footnote{This can be shown, for example, using the fact that \(c_1(\cL_{\det})\) is a derivation with respect to the monoidal structure on \(\Sat_{H,\hbar}\), together with the description of its reduction modulo \(\hbar\) used in the proof of Proposition \ref{prop:localunit}.} A straightforward computation shows that \[\tr_{\Sat_{H,\hbar}}(c_1(\cL_{\det}),U_{\hbar}(\frhc^{\shear})\otimes V_{\lambda_H})=-\frac{\dim(\frhc)}{2}\epsilon_{\lambda}\hbar.\] This concludes the proof.

\end{proof}

\subsection{Global special cohomological correspondences}
Now we turn to global special cohomological correspondences. 
Define the period sheaf
\begin{equation}\label{eq:periodsheaf}
\cP_X:=\Delta_!\uk_{\Bun_H}\in \Shv(\Bun_G),
\end{equation}
where $\Delta:\Bun_H\to \Bun_G$ is the diagonal morphism.

\subsubsection{Definition}
Given a correspondence between Artin stacks \[B \xleftarrow{\lh} C \xrightarrow{\rh} A\] and sheaves $\cF\in\Shv(A), \cG\in \Shv(B), \cK\in \Shv(C)$, we define the space of cohomological correspondences between $\cF$ and $\cG$ with kernel $\cK$ to be the vector space \[\Cor_{C,\cK}(\cF,\cG):=\Hom^0(\rh^*\cF \otimes \cK,\lh^!\cG).\] We refer to \cite[\S4.2]{liu2025higherperiodintegralsderivatives} and \cite[\S2.2]{wang2026arithmeticvolumeshtukaslanglands} for operations on cohomological correspondences with kernels.

A \emph{global special cohomological correspondence} of degree $d$ with kernel $\cK_I\in \Shv(\Hk_{G,I})$ is an element \[\frc_{\cK_I} \in \Cor_{\Hk_{G,I},\cK_I}(\cP_X\boxtimes\uk_{C^I},\cP_X\boxtimes\uk_{C^I}\langle d\rangle).\]

Given two finite sets $I$ and $J$ with $|I|=r$ and $|J|=s$, and given two sheaves $\cK_I\in \Shv(\Hk_{G,I})$ and $\cK_J\in \Shv(\Hk_{G,J})$, there is an obvious operation 
\begin{equation} \begin{split}
    \boxcirc:\Cor_{\Hk_{G,I},\cK_I}(\cP_X\boxtimes\uk_{C^I},\cP_X\boxtimes\uk_{C^I})\otimes \Cor_{\Hk_{G,J},\cK_J}(\cP_X\boxtimes\uk_{C^J},\cP_X\boxtimes\uk_{C^J}) \\ \to \Cor_{\Hk_{G,I\sqcup J},\cK_I\widetilde{\boxtimes} \cK_J}(\cP_X\boxtimes\uk_{C^{I\sqcup J}},\cP_X\boxtimes \uk_{C^{I\sqcup J}})
\end{split}.\end{equation}

When $\cK_I=\Sat(V^I)$ and $\cK_J=\Sat(V^J)$ for $V^I\in\Rep(\Gc^I)$ and $V^J\in\Rep(\Gc^J)$, under the canonical isomorphism \[\Cor_{\Hk_{G,I\sqcup J},\Sat(V^I\boxtimes V^J)}(\cP_X\boxtimes\uk_{C^{I\sqcup J}},\cP_X\boxtimes \uk_{C^{I\sqcup J}})\cong \Cor_{\Hk_{G,J\sqcup I},\Sat(V^J\boxtimes V^I)}(\cP_X\boxtimes\uk_{C^{I\sqcup J}},\cP_X\boxtimes \uk_{C^{I\sqcup J}}), \] as in \cite[Conjecture\,4.45]{liu2025higherperiodintegralsderivatives}, one has the commutator operation \begin{equation}\begin{split}
    [~,~]:\Cor_{\Hk_{G,I},\Sat(V^I)}(\cP_X\boxtimes\uk_{C^I},\cP_X\boxtimes\uk_{C^I})\otimes \Cor_{\Hk_{G,J},\Sat(V^J)}(\cP_X\boxtimes\uk_{C^J},\cP_X\boxtimes\uk_{C^J}) \\  \to \Cor_{\Hk_{G,I\sqcup J},\Sat(V^I\boxtimes V^J)}(\cP_X\boxtimes\uk_{C^{I\sqcup J}},\cP_X\boxtimes \uk_{C^{I\sqcup J}})\end{split}.
\end{equation} defined by \[[\frc_{V^I},\frc_{V^J}]=\frc_{V^I}\boxcirc \frc_{V^J}-\frc_{V^J}\boxcirc \frc_{V^I}\] for $\frc_{V^I}\in \Cor_{\Hk_{G,I},\Sat(V^I)}(\cP_X\boxtimes\uk_{C^I},\cP_X\boxtimes\uk_{C^I})$ and $\frc_{V^J}\in \Cor_{\Hk_{G,J},\Sat(V^J)}(\cP_X\boxtimes\uk_{C^J},\cP_X\boxtimes\uk_{C^J})$.

\subsubsection{Local-to-global procedure}
There is a canonical morphism \begin{equation}\label{eq:globalization}\glob:\Hom^0(V,\PL_{X,\hbar}) \to \Cor_{\Hk_{G},\Sat(V)}(\cP_X\boxtimes\uk_C,\cP_X\boxtimes\uk_C)\end{equation} sending a local special cohomological correspondence to a global special cohomological correspondence as defined in \cite[Definition\,4.40]{liu2025higherperiodintegralsderivatives}. This procedure is compatible with composition of cohomological correspondences in the sense that for $\frc_V^{\loc}\in \Hom^0(V,\PL_{X,\hbar})$ and $\frc_W^{\loc}\in \Hom^0(W,\PL_{X,\hbar})$, one has \begin{equation}\label{eq:ltogassoc}
    \glob(\frc_V^{\loc}\circ \frc_W^{\loc})=\glob(\frc_V^{\loc})\circ \glob(\frc_W^{\loc}).
\end{equation}

\subsubsection{Examples of global special cohomological correspondences}

Consider the diagram \begin{equation}
    \begin{tikzcd}
        \Bun_H\times C \ar[d, "\D\times\id"] & \Hk_{H,\leq\lambda_H} \ar[l] \ar[r] \ar[d, "\D^{\Hk}"] & \Bun_H \times C \ar[d, "\D\times\id"] \\
        \Bun_G\times C & \Hk_{G,\leq\lambda} \ar[l] \ar[r]  & \Bun_G \times C
    \end{tikzcd}
\end{equation}
Define the \emph{global diagonal cohomological correspondence} \begin{equation}\label{eq:diagcorglob}\frd_{\lambda}:=\D^{\Hk}_![\Hk_{H,\leq\lambda_H}/\Bun_H\times C]\in \Cor_{\Hk_{G},\IC_{\lambda}}(\cP_X\boxtimes\uk_C,\cP_X\boxtimes\uk_C)\end{equation} as considered in \cite[\S1.3.2]{wang2025specialcycleshtukascategorical}.

Similarly, given a sequence of coweights $\lambda_{H,I}=(\lambda_{H,i})_{i\in I}\in X_*(T_H)^I_+$ for the finite set $I=\{1,2,\dots,r\}$, define $\lambda_I=(\lambda_{H,I},\lambda_{H,I})\in X_*(T)^I_+$.

One considers the diagram \begin{equation}
    \begin{tikzcd}
        \Bun_H\times C^I \ar[d, "\D\times\id"] & \Hk_{H,\leq\lambda_{H,I}} \ar[l] \ar[r] \ar[d, "\D^{\Hk}"] & \Bun_H \times C^I \ar[d, "\D\times\id"] \\
        \Bun_G\times C^I & \Hk_{G,\leq\lambda_I} \ar[l] \ar[r]  & \Bun_G \times C^I
    \end{tikzcd}
\end{equation} and defines the multi-leg version \begin{equation}\label{eq:diagcorglobmulti}
    \frd_{\lambda_I}:=\D^{\Hk}_![\Hk_{H,\leq\lambda_{H,I}}/\Bun_H\times C^I]\in \Cor_{\Hk_{G,I},\IC_{\lambda_I}}(\cP_X\boxtimes\uk_{C^I},\cP_X\boxtimes\uk_{C^I}).
\end{equation}


Define the \emph{global Lefschetz cohomological correspondence}

\begin{equation}\label{eq:lefcorglob}
    \begin{split}
\frd_{\det,\lambda}:=\D^{\Hk}_!([\Hk_{H,\leq\lambda_H}/\Bun_H\times C]\cup c_1(\cL_{\det})) \\ \in \Cor_{\Hk_{G},\IC_{\lambda}}(\cP_X\boxtimes\uk_C,\cP_X\boxtimes\uk_C\langle 2\rangle) \end{split}
\end{equation}
and its multi-leg version
\begin{equation}\label{eq:lefcorglobmulti} \begin{split}
\frd_{\det,\lambda_I}:=\D^{\Hk}_!([\Hk_{H,\leq\lambda_{H,I}}/\Bun_H\times C^I]\cup \prod_{i=1}^r c_1(\cL_{\det,i})) \\ \in \Cor_{\Hk_{G,I},\IC_{\lambda_I}}(\cP_X\boxtimes\uk_{C^I},\cP_X\boxtimes\uk_{C^I}\langle 2r\rangle)
.\end{split}\end{equation} 

It is easy to see that 
\begin{equation}
    \glob(\frd_{\lambda}^{\loc}) = \frd_{\lambda}.
\end{equation} 
\begin{equation}
    \glob(\frd_{\det,\lambda}^{\loc}) = \frd_{\det,\lambda}.
\end{equation}

Define the \emph{global adjoint cohomological correspondence} to be \begin{equation}\label{eq:adjointcorglob}
    \fra\frd:=\glob(\fra\frd^{\loc})  \in \Cor_{\Hk_{G},\Sat(\frhc\boxtimes k)}(\cP_X\boxtimes\uk_C,\cP_X\boxtimes\uk_C\langle 2\rangle) .
\end{equation}

\subsubsection{A global commutator relation}

Consider the map \begin{equation}
    \D_{C,!}:\Cor_{\Hk_G,\Sat(\frhc\boxtimes k)*\IC_{\lambda}}(\cP_X\boxtimes\uk_C,\cP_X\boxtimes\uk_C)\to \Cor_{\Hk_{G,\{1,2\}},\Sat(\frhc\boxtimes k)\widetilde{\boxtimes}\IC_{\lambda}}(\cP_X\boxtimes\uk_{C^2},\cP_X\boxtimes \uk_{C^2}\langle 2\rangle) .\end{equation}

\begin{prop}\label{prop:globalcommutator}
    Assuming $g\neq 1$, the following identity holds \[[\fra\frd,\frd_{\lambda}]=\D_{C,!}((\act_{V_{\lambda_H}}\boxtimes \id)^*\frd_{\lambda}).\]
\end{prop}

\begin{proof}
    This follows directly from \cite[Theorem\,4.47]{liu2025higherperiodintegralsderivatives} and the local commutator relation in Proposition \ref{prop:localcommutator}. Note that \cite[Assumption\,4.46(1)]{liu2025higherperiodintegralsderivatives} is satisfied in our case, and the condition $g\neq 1$ is \cite[Assumption\,4.46(2)]{liu2025higherperiodintegralsderivatives}.
\end{proof}

\section{Geometric period integral}\label{sec:geometricperiod}
In this section, we study the geometric period integral of a Hecke eigensheaf in the group case. The key technical result is Theorem \ref{thm:clifford}, which constructs a Clifford algebra action on the resulting vector space. Moreover, in \S\ref{sec:intobs}, we introduce the intersection observable and compute it explicitly, which will play a central role in the study of the intersection number of diagonal cycles in \S\ref{sec:diagcycle}.

\subsection{Definition of the geometric period integral}\label{sec:geometricperioddef}
Consider the map $\D:\Bun_H\to \Bun_G$ given by the diagonal embedding. Define the geometric period integral to be the functor \begin{equation}
    \int_X: \Shv(\Bun_G)\to \Vect,~\cF\mapsto \Gamma_c(\Bun_H, \D^*\cF).
\end{equation}

Moreover, for each $e\in \pi_1(H)\cong \pi_0(\Bun_H)$, define the $e$-component of the geometric period integral to be the functor \begin{equation}
    \int_{X,e}:\Shv(\Bun_G)\to \Vect,~\cF\mapsto \Gamma_c(\Bun_H^e, \D^*\cF).
\end{equation}

We use the notations in \S\ref{sec:cohshtuka}. Fix a geometrically strongly irreducible $\Hc$-local system $\s_H\in \Loc_{\Hc}^{\arith}(k)$ and let $\s=(\s_H, c_{\Hc}(\s_H)) \in \Loc_{\Gc}^{\arith}(k)$. Let $\cF_{\s_H} \in \Shv(\Bun_H)$ be a Weil Hecke eigensheaf with eigenvalue $\s_H$ that is perverse irreducible on each connected component of $\Bun_H$, and take \begin{equation}\label{eq:eigensheaf}\cF_{\s}:=\cF_{\s_H}\boxtimes\mathbb{D}^{\mathrm{Ver}}(\cF_{\s_H})\in \Shv(\Bun_G),\end{equation} which is a Hecke eigensheaf with eigenvalue $\s$. 

\subsubsection{The graded dimension}

We first study the graded dimension of the geometric period integral.

\begin{prop}\label{prop:dim}
    For each $e\in \pi_1(H)$ and $i\in \ZZ$, we have \[\dim H^{i}(\int_{X,e}\cF_{\s})=\dim  \wedge^{-i}H^1(C,\frhc_{\s_H}).\]
\end{prop}

\begin{proof}
    Ignoring the Tate twist, we have \[\begin{split} & H^i(\int_{X,e}\cF_{\s}) \\ = & H_c^i(\Bun_H^e, \cF_{\s_H}\otimes\mathbb{D}^{\mathrm{Ver}}(\cF_{\s_H})) \\ \cong & H^{i+(2g-2)\dim H}(\Bun_H^e, \cF_{\s_H}\overset{!}{\otimes} {\mathbb{D}}^{\mathrm{Ver}}(\cF_{\s_H})) \\ \cong & \Hom^{i+(2g-2)\dim H}(\cF_{\s_H}|_{\Bun_H^e}, \cF_{\s_H}|_{\Bun_H^e}) \\ \cong & \wedge^{i+(2g-2)\dim H}H^1(C,\frhc_{\s_H}) \\ \cong & \wedge^{-i}H^1(C,\frhc_{\s_H})  \end{split}.\] Here, the second isomorphism follows from \cite[Theorem\,0.4.5]{gaitsgory2017strange} and \cite[Corollary\,3.3.7]{arinkin2022dualityautomorphicsheavesnilpotent}, the fourth isomorphism follows from the fact that $\cF_{\s_H}$ corresponds to a skyscraper sheaf on $\Loc_{\Hc}^{\res}(k)$ supported at $\s_H$ under the geometric Langlands correspondence and \cite[Theorem\,0.1.4]{gaitsgory2025geometriclanglandspositivecharacteristic}.
\end{proof}

\subsubsection{The canonical element}
For each $e\in \pi_1(H)$, there is a natural isomorphism \begin{equation}
    H^0(\int_{X,e}\cF_{\s})=H_c^0(\Bun_H^e, \cF_{\s_H}\otimes\mathbb{D}^{\mathrm{Ver}}(\cF_{\s_H}))\isom H^0_c(\Bun_H^e, \omega_{\Bun_H^e})\isom k
.\end{equation}  We use \begin{equation}\label{eq:canelement}
    \eta_{e,\s}\in H^0(\int_{X,e}\cF_{\s})\end{equation} to denote the element corresponding to $1\in k$ under the above isomorphism.

\subsection{\texorpdfstring{$L$}{L}-observables}
Now we introduce $L$-observables in the group case.
\subsubsection{Definition}\label{sec:lobsdef}
Following \cite[\S18]{BZSV}, we consider the algebra of $L$-observables \begin{equation}
    \mathbb{O}_{X,\s}=\End(\int_X\cF_{\s})\in \Alg(\Vect)
.\end{equation} On the level of cohomologies, we have \begin{equation}
    H^*\mathbb{O}_{X,\s}\cong \End(H^*(\int_X\cF_{\s})).
\end{equation} 

\subsubsection{Gradings by connected components}
Note that \begin{equation}\mathbb{O}_{X,\s}=\bigoplus_{(e,e')\in \pi_1(H)^2}\mathbb{O}_{X,(e,e'),\s}\end{equation} where $\mathbb{O}_{X,(e,e'),\s}=\Hom(\int_{X,e'}\cF_{\s}, \int_{X,e}\cF_{\s})$. We write \begin{equation}\mathbb{O}_{X,e,\s}=\mathbb{O}_{X,(e,e),\s}\end{equation} for $e\in \pi_1(H)$.

\subsubsection{\texorpdfstring{$L$}{L}-observables from special cohomological correspondences}
We study $L$-observables arising from special cohomological correspondences.

Given a cohomological correspondence
$\frc_{V^I}\in \Cor_{\Hk_{G,I},\Sat(V^I)}
(\cP_X\boxtimes\uk_{C^I},\cP_X\boxtimes\uk_{C^I})$,
there is an associated map
\begin{equation}
\begin{aligned}
\frc_{V^I,\s}
&\in \Hom^0\left(
V_{\s}^I\otimes\int_X\cF_{\s},
\uk_{C^I}\otimes\int_X\cF_{\s}
\right) \\
&\cong \Hom^0\left(
\Gamma(C^I,V_{\s}^I)\langle 2r\rangle\otimes\int_X\cF_{\s},
\int_X\cF_{\s}
\right) \\
&\cong \Hom^0\left(
\Gamma(C^I,V_{\s}^I)\langle 2r\rangle,
\mathbb{O}_{X,\s}
\right)
\end{aligned}
\end{equation}
as considered in \cite[\S6.2.1]{liu2025higherperiodintegralsderivatives}.

This construction is associative in the following sense: When $V^I=\boxtimes_{i=1}^r V_i$ for $V_i\in \Rep(\Gc)$, and $\frc_{V^I}=\frc_{V_1}\boxcirc \cdots \boxcirc \frc_{V_r}$ for $\frc_{V_i}\in \Cor_{\Hk_{G},\Sat(V_i)}(\cP_X\boxtimes\uk_C,\cP_X\boxtimes\uk_C)$, one has \begin{equation}\label{eq:Lobsassoc}
    \frc_{V^I,\s}=\rmm\circ (\frc_{V_1,\s}\otimes \cdots \otimes \frc_{V_r,\s})\in \Hom^0(\Gamma(C^I,V_{\s}^I)\langle 2r\rangle, \mathbb{O}_{X,\s})
\end{equation} where $\rmm:\mathbb{O}_{X,\s}^{\otimes r}\to \mathbb{O}_{X,\s}$ is the multiplication map.

Moreover, the construction is compatible with fusion in the following sense: Given $V^I\in \Rep(\Gc^I)$ and $\frc_{V^I}\in \Cor_{\Hk_{G,I},\Sat(V^I)}(\cP_X\boxtimes\uk_{C^I},\cP_X\boxtimes\uk_{C^I})$. Consider $\D^*_{\Gc}V^I\in \Rep(\Gc)$ and \[\frc_{\D^*_{\Gc}V^I}=\D_C^*\frc_{V^I}\in \Cor_{\Hk_{G},\Sat(\D^*_{\Gc}V^I)}(\cP_X\boxtimes\uk_C,\cP_X\boxtimes\uk_C)\] where $\D_{\Gc}:\Gc\to \Gc^I$ and $\D_C:C\to C^I$ are the diagonal maps. Then \begin{equation}\label{eq:Lobsfusion}
    \frc_{\D^*_{\Gc}V^I,\s}=\frc_{V^I,\s}\circ \D_{C,!}\in \Hom^0(\Gamma(C,\D^*_{C}V^I_{\s})\langle 2\rangle , \mathbb{O}_{X,\s})
\end{equation} where $\D_{C,!}:\Gamma(C,\D^*_{C}V_{\s}^I)\langle 2\rangle \to \Gamma(C^I,V_{\s}^I)\langle 2r\rangle$ is the Gysin homomorphism.

Combining \eqref{eq:Lobsassoc} and \eqref{eq:Lobsfusion}, note that $\D_C^*(\frc_{V_1}\boxcirc \cdots \boxcirc \frc_{V_r})=\frc_{V_1} \circ \cdots \circ \frc_{V_r}$, we have \begin{equation}\label{eq:Lobsassocfusion}
    (\frc_{V_1}\circ \cdots \circ \frc_{V_r})_{\s}=\rmm\circ (\frc_{V_1,\s}\otimes \cdots \otimes \frc_{V_r,\s}) \circ \D_{C,!}\in \Hom^0(\Gamma(C,\bigotimes_{i=1}^r V_{i,\s})\langle 2\rangle , \mathbb{O}_{X,\s}).
\end{equation}

\subsubsection{Examples of \texorpdfstring{$L$}{L}-observables}

For each $\lambda=(\lambda_H,\lambda_H)\in X_*(T)_+$, let $e_{\lambda}\in\pi_1(H)$ be the image of $\lambda_H$ under the natural map $X_*(T_H)\to \pi_1(H)$. From \eqref{eq:diagcorglob}, we obtain the diagonal $L$-observables \begin{equation}\label{eq:diagLobs}
    \frd_{\lambda,\s}:\Gamma(C,V_{\lambda,\s})\langle 2\rangle\to \bigoplus_{e\in\pi_1(H)}\mathbb{O}_{X,(e+e_{\lambda},e),\s}\sub  \mathbb{O}_{X,\s}
.\end{equation} From \eqref{eq:lefcorglob}, we obtain the Lefschetz $L$-observable \begin{equation}\label{eq:lefLobs}
    \frd_{\det,\lambda,\s}:\Gamma(C,V_{\lambda,\s})\to \bigoplus_{e\in\pi_1(H)}\mathbb{O}_{X,(e+e_{\lambda},e),\s} \sub \mathbb{O}_{X,\s}.
\end{equation}
From \eqref{eq:adjointcorglob}, we obtain the adjoint $L$-observable \begin{equation}\label{eq:adjointLobs}
    \fra\frd_{\s}:\Gamma(C,\frhc_{\s_H})\to \bigoplus_{e\in\pi_1(H)}\mathbb{O}_{X,e,\s} \sub \mathbb{O}_{X,\s}.
\end{equation}

\subsubsection{Translation \texorpdfstring{$L$}{L}-observables}
Now we introduce translation $L$-observables, which canonically identify the cohomologies of the geometric period integrals on different connected components of $\Bun_H$.
\begin{defn}\label{defn:translationLobs}
    For $e,e'\in \pi_1(H)$, an element $1_{e,e'}\in H^0\mathbb{O}_{X,(e,e'),\s}$ is called a translation $L$-observable (from $e'$ to $e$) if it satisfies the following conditions:
    \begin{enumerate}
        \item $1_{e,e'}(\eta_{e',\s})=\eta_{e,\s}$;
        \item $1_{e,e'}$ commutes with diagonal $L$-observables \eqref{eq:diagLobs}, that is, for each $\lambda=(\lambda_H,\lambda_H)\in X_*(T)_+$ such that $e_{\lambda}=0$, we have \[1_{e,e'}\circ H^*(\frd_{\lambda,\s})=H^*(\frd_{\lambda,\s})\circ 1_{e,e'} \in \Hom^0( H^{*+2}(C,V_{\lambda,\s})(1), H^*\mathbb{O}_{X,(e,e'),\s});\]
        \item When $e=e'$, we require that $1_{e,e'}\in H^0\mathbb{O}_{X,e,\s}$ is the unit.
    \end{enumerate}
\end{defn}

\begin{prop}\label{prop:trans}
    For every pair $e,e'\in \pi_1(H)$, the translation $L$-observable $1_{e,e'}\in H^0\mathbb{O}_{X,(e,e'),\s}$ exists and is unique. Moreover, it induces an isomorphism $1_{e,e'}:H^*(\int_{X,e'}\cF_{\s})\isom H^*(\int_{X,e}\cF_{\s})$.
\end{prop}

\begin{proof}
    The uniqueness is a consequence of Corollary \ref{cor:AB}. The existence of $1_{e,e}$ is obvious, and the general case follows from Proposition \ref{prop:diagh0}.
\end{proof}

Given Proposition \ref{prop:trans}, the span of translation $L$-observables $\{1_{e,e'}\}_{e,e'\in \pi_1(H)}$ is a subalgebra of $H^0\mathbb{O}_{X,\s}$ isomorphic to $\End(k[\pi_1(H)])$. Moreover, the translation $L$-observables canonically identify the algebras $H^*\mathbb{O}_{X,e,\s}$ for different $e\in \pi_1(H)$, and we define the reduced algebra of $L$-observables to be \begin{equation}\label{eq:reducedLobs}
    H^*\mathbb{O}_{X,\s}^{\red}:=H^*\mathbb{O}_{X,e,\s}
\end{equation} for any $e\in \pi_1(H)$. We have a canonical isomorphism of algebras \begin{equation}\label{eq:transiso}
     H^*\mathbb{O}_{X,\s}^{\red} \otimes \End(k[\pi_1(H)])\cong H^*\mathbb{O}_{X,\s}.
\end{equation}

For $e_0\in \pi_1(H)$, we define the translation by $e_0$ observable to be \begin{equation}\label{eq:shiftobs}
     T_{e_0}:=\sum_{e\in \pi_1(H)}1_{e+e_0,e} \in \End(k[\pi_1(H)])\sub  H^0\mathbb{O}_{X,\s}.
    \end{equation}

\subsubsection{Clifford relation}
Now we study the commutator between the diagonal $L$-observables \eqref{eq:diagLobs} and the adjoint $L$-observables \eqref{eq:adjointLobs}.

The commutator in the algebra $H^*\mathbb{O}_{X,\s}$ induces a map \begin{equation}\label{eq:commutatorLobs}
    [H^1(\fra\frd_{\s}),H^{-1}(\frd_{\lambda,\s})]:H^1(C,\frhc_{\s_H})\otimes H^1(C,V_{\lambda,\s})(1)\to H^0\mathbb{O}_{X,\s}. \end{equation}
The canonical element $\counit_{\frhc}: \frhc\otimes \frhc^*\to k$ induces a natural pairing via cup product \begin{equation}\label{eq:couniths}\counit_{H^1(C,\frhc_{\s_H})}: H^1(C,\frhc_{\s_H}) \otimes H^1(C,\frhc_{\s_H}^*)(1) \to H^2(C,\uk_C)(1) \cong k.\end{equation}

Consider the map \begin{equation}\label{eq:act*}
    \act_{V_{\lambda_H}}':V_{\lambda_H}\otimes V_{-w_0\lambda_H}\xrightarrow{\unit_{\frhc}\otimes\id\otimes\id} \frhc^*\otimes \frhc\otimes V_{\lambda_H}\otimes V_{-w_0\lambda_H}\xrightarrow{\id\otimes\act_{V_{\lambda_H}} \otimes \id} \frhc^*\otimes V_{\lambda_H}\otimes V_{-w_0\lambda_H}\xrightarrow{\id\otimes\counit_{V_{\lambda_H}}} \frhc^* 
.\end{equation} It induces a natural map \begin{equation}\label{eq:act*s} H^1(\act_{V_{\lambda_H},\s_H}'):H^{1}(C, V_{\lambda,\s})(1)\to H^1(C,\frhc_{\s_H}^*)(1).\end{equation}

\begin{thm}\label{thm:clifford2}
    The following identity holds \[\begin{split}
        [H^1(\fra\frd_{\s}), H^{-1}(\frd_{\lambda,\s})]=\counit_{H^1(C,\frhc_{\s_H})}\circ (\id\otimes H^1(\act_{V_{\lambda_H},\s_H}'))\cdot T_{e_{\lambda}} 
        \\ \in \Hom^0(H^1(C,\frhc_{\s_H})\otimes H^1(C,V_{\lambda,\s})(1), H^0\mathbb{O}_{X,\s})\end{split}\] 
        where $T_{e_{\lambda}}\in H^0\mathbb{O}_{X,\s}$ is defined in \eqref{eq:shiftobs}.
\end{thm}

We will prove Theorem \ref{thm:clifford2} in \S\ref{sec:relproof}.

When $\lambda\neq 0$, we choose a splitting $H^1(C,V_{\lambda,\s})(1)=H^1(C,\frhc_{\s_H}^*)(1)\oplus K$ for the quotient map \[H^1(\act_{V_{\lambda_H},\s_H}'):H^1(C,V_{\lambda,\s})(1)\to H^1(C,\frhc_{\s_H}^*)(1).\] We obtain the following immediate corollary of Theorem \ref{thm:clifford2}.

\begin{cor}\label{cor:clifford2split}
For $\lambda\neq 0$, we have \[\begin{split}[H^1(\fra\frd_{\s}), H^{-1}(\frd_{\lambda,\s})|_{H^1(C,\frhc_{\s_H}^*)(1)}]=\counit_{H^1(C,\frhc_{\s_H})}\cdot T_{e_{\lambda}}  \\ \in \Hom^0( H^1(C,\frhc_{\s_H})\otimes H^1(C,\frhc_{\s_H}^*)(1), H^0\mathbb{O}_{X,\s})\end{split}.\] 
\end{cor}

\subsubsection{Atiyah--Bott formula}
We also have the following corollary.
\begin{cor}\label{cor:AB}
    When $e_{\lambda}=0$ and $\lambda\neq 0$, the map $H^{-1}(\frd_{\lambda,\s})|_{H^1(C,\frhc^*_{\s_H})(1)}$ induces an isomorphism of vector spaces \[\wedge^*H^{-1}(\frd_{\lambda,\s})|_{H^1(C,\frhc^*_{\s_H})(1)}\cdot \eta_{e,\s}: \wedge^{*}(H^1(C,\frhc_{\s_H}^*)(1))\isom H^{-*}(\int_{X,e}\cF_{\s})\] by taking the action on $\eta_{e,\s}\in H^0(\int_{X,e}\cF_{\s})$ for each $e\in \pi_1(H)$.
\end{cor}
\begin{proof}
    Consider the induced map on the tensor algebra \[H^{-1}(\frd_{\lambda,\s})|_{H^1(C,\frhc^*_{\s_H})(1)}^{\otimes *}\cdot \eta_{e,\s}: (H^1(C,\frhc_{\s_H}^*)(1))^{\otimes *}\to H^{-*}(\int_{X,e}\cF_{\s}).\] It factors through the natural quotient map $(H^1(C,\frhc_{\s_H}^*)(1))^{\otimes *}\to \wedge^*(H^1(C,\frhc_{\s_H}^*)(1))$ by Corollary \ref{cor:clifford2split} (or more directly, follows from the fact that diagonal $L$-observables are mutually commutative which is explained in the proof of Proposition \ref{prop:filtrationvialoc}(1)). Therefore, we obtain the map $\wedge^*H^{-1}(\frd_{\lambda,\s})|_{H^1(C,\frhc^*_{\s_H})(1)}\cdot \eta_{e,\s}: \wedge^{*}(H^1(C,\frhc_{\s_H}^*)(1))\to H^{-*}(\int_{X,e}\cF_{\s})$. The injectivity of this map follows from the $e_{\lambda}=0$ case of Corollary \ref{cor:clifford2split} (see the proof of \cite[Lemma\,6.16]{wang2025specialcycleshtukascategorical}). The surjectivity follows from a dimension comparison using Proposition \ref{prop:dim}.
\end{proof}

We also note the following consequence of Corollary \ref{cor:AB} for later use.
\begin{cor}\label{cor:duality}
    For each $e, e'\in \pi_1(H)$, there exists an isomorphism of graded vector spaces \[(H^*(\int_{X,e}\cF_{\s}))^*\cong H^{*-(2g-2)\dim H}(\int_{X,e'}\cF_{\s})(-(g-1)\dim H) .\]
\end{cor}




\subsubsection{Diagonal \texorpdfstring{$L$}{L}-observables}
Now we study further the diagonal $L$-observables defined in \eqref{eq:diagLobs} for each $\lambda=(\lambda_H,\lambda_H)\in X_*(T)_+$.

We first study $H^{-1}(\frd_{\lambda,\s})$.
\begin{prop}\label{prop:diagh-1}
    For $\lambda\neq 0$, the map \[H^{-1}(\frd_{\lambda,\s}):H^1(C,V_{\lambda,\s})(1)\to H^{-1}\mathbb{O}_{X,\s}\] factors through $H^1(\act_{V_{\lambda_H},\s_H}')$ defined in \eqref{eq:act*s}. That is, there exists a unique map \begin{equation}\label{eq:diagLobsq}
    \overline{H^{-1}(\frd_{\lambda,\s})}:H^1(C,\frhc_{\s_H}^*)(1)\to H^{-1}\mathbb{O}_{X,\s}
\end{equation} such that $H^{-1}(\frd_{\lambda,\s})=\overline{H^{-1}(\frd_{\lambda,\s})}\circ H^1(\act_{V_{\lambda_H},\s_H}')$.
\end{prop}
\begin{proof}
    Consider $K=\ker(H^1(\act_{V_{\lambda_H},\s_H}'))\sub H^1(C,V_{\lambda,\s})(1)$. We need to show that $H^{-1}(\frd_{\lambda,\s})(K)=0$. By Theorem \ref{thm:clifford2}, we know that $H^{-1}(\frd_{\lambda,\s})(K)\sub H^{-1}(\mathbb{O}_{X,\s})$ commutes with adjoint $L$-observables. This implies that $H^{-1}(\frd_{\lambda,\s})(K)$ lies in the subalgebra of $H^*\mathbb{O}_{X,\s}$ generated by adjoint $L$-observables and translation $L$-observables (use Corollary \ref{cor:AB} and Corollary \ref{cor:clifford2split}), hence, implies that $H^{-1}(\frd_{\lambda,\s})(K)=0$ for cohomological degree reasons.
\end{proof}

Combining Theorem \ref{thm:clifford2} and Proposition \ref{prop:diagh-1}, we obtain the following Clifford relation between the diagonal $L$-observables and the adjoint $L$-observables.

\begin{thm}\label{thm:clifford}
    For $\lambda\neq 0$, the following identity holds \[[H^1(\fra\frd_{\s}), \overline{H^{-1}(\frd_{\lambda,\s})}]=\counit_{H^1(C,\frhc_{\s_H})}\cdot T_{e_{\lambda}}  \in \Hom^0(H^1(C,\frhc_{\s_H}) \otimes H^1(C,\frhc_{\s_H}^*)(1), H^0\mathbb{O}_{X,\s})\] where $T_{e_{\lambda}}\in H^0\mathbb{O}_{X,\s}$ is defined in \eqref{eq:shiftobs}. 
\end{thm}

Then we study $H^0(\frd_{\lambda,\s})$. Consider the map \begin{equation}\label{eq:trivialLobs}
    H^2(\counit_{V_{\lambda_H},\s_H}): H^2(C,V_{\lambda,\s})(1) \to H^2(C, \uk_C)(1)\cong k
.\end{equation}

\begin{prop}\label{prop:diagh0}
    The following identity holds \[H^0(\frd_{\lambda,\s})=H^2(\counit_{V_{\lambda_H},\s_H})\cdot T_{e_{\lambda}}  \in \Hom^0(H^2(C,V_{\lambda,\s})(1), H^0\mathbb{O}_{X,\s}).\] 
\end{prop}

\begin{proof}
    We need to show that each component of $H^0(\frd_{\lambda,\s})$ applied to appropriate elements in $H^2(C, V_{\lambda,\s})(1)$ satisfies the condition in Definition \ref{defn:translationLobs}.
    
    The condition Definition \ref{defn:translationLobs}(1) follows from \cite[Lemma\,5.15]{wang2025specialcycleshtukascategorical}. The condition Definition \ref{defn:translationLobs}(2) follows from the fact that diagonal $L$-observables are mutually commutative (see the proof of Proposition \ref{prop:filtrationvialoc}(1)). The condition Definition \ref{defn:translationLobs}(3) follows from (1)(2) by Corollary \ref{cor:AB}.
\end{proof}

\subsubsection{Reduced \texorpdfstring{$L$}{L}-observables}
Now we show that the $L$-observables can usually be reduced to the reduced algebra of $L$-observables defined in \eqref{eq:reducedLobs}.

Let $V\in \Rep(\Gc)^{\heartsuit}$ be an irreducible representation whose central character defines $(e, e)\in \pi_1(H)^2$. For any $\frc^{\loc}_V\in \Hom^0_{\Rep(\Gc)}(V, \PL_{X,\hbar}\langle d\rangle)$ with $0\leq d\leq 2$, consider the induced $L$-observables \[H^*(\glob(\frc^{\loc}_V)_{\s}):H^{*-d+2}(C,V_{\s})(-d/2+1)\to H^*\mathbb{O}_{X,\s}.\]

\begin{prop}\label{prop:reducedLobs}
    There exists a unique map \begin{equation}\label{eq:redobs}
        H^*(\glob(\frc^{\loc}_V)_{\s})^{\red}: H^{*-d+2}(C,V_{\s})(-d/2+1) \to H^*\mathbb{O}_{X,\s}^{\red}\end{equation} which we call the reduced $L$-observable, such that under the isomorphism \eqref{eq:transiso}, we have \[H^*(\glob(\frc^{\loc}_V)_{\s})=H^*(\glob(\frc^{\loc}_V)_{\s})^{\red}\otimes T_{e}.\]
\end{prop}
\begin{proof}
    We only need to show that the image of $H^*(\glob(\frc^{\loc}_V)_{\s})$ commutes with $T_e\in H^*\mathbb{O}_{X,\s}$ for all $e\in \pi_1(H)$. The statement is trivial when $g=1$. When $g\neq 1$, by \cite[Theorem\,4.47]{liu2025higherperiodintegralsderivatives}, we know 
    \[[[H^*(\glob(\frc^{\loc}_V)_{\s}), H^*(\frd_{\lambda,\s})],H^*(\frd_{\lambda',\s})]=0\] for any $\lambda_H,\lambda'_H\in X_*(T_H)_+$. By Proposition \ref{prop:diagh0}, this implies that 
    \[[[H^*(\glob(\frc^{\loc}_V)_{\s}), T_e],T_{e'}]=0\]
    for all $e,e'\in \pi_1(H)$. This implies the desired commutativity.
\end{proof}

In particular, we have the reduced diagonal, Lefschetz, and adjoint $L$-observables:
\begin{equation}\label{eq:reddiag}
    H^*(\frd_{\lambda,\s})^{\red}:H^{*+2}(C, V_{\lambda,\s})(1) \to H^*\mathbb{O}_{X,\s}^{\red}
\end{equation}

\begin{equation}\label{eq:redlef}
    H^*(\frd_{\det,\lambda,\s})^{\red}:H^{*}(C, V_{\lambda,\s}) \to H^*\mathbb{O}_{X,\s}^{\red}
\end{equation}
\begin{equation}\label{eq:redad}
    H^*(\fra\frd_{\s})^{\red}:H^*(C,\frhc_{\s_H}) \to H^*\mathbb{O}_{X,\s}^{\red}.
\end{equation}

\subsection{Ran filtration}
In this section, we introduce a natural filtration $\{F_{\bullet}^{\Ran}H^*\mathbb{O}_{X,\s}^{\red}\}_{\bullet\in \ZZ}$ on the reduced algebra of $L$-observables $H^*\mathbb{O}_{X,\s}^{\red}$, which we call the Ran filtration.\footnote{The name ``Ran filtration" is taken from \cite[\S5.5.4]{FYZvolume} in a different but related context.}

\subsubsection{Definition of the filtration}\label{sec:deffil}

We introduce a filtration on $H^*\mathbb{O}_{X,e,\s}$ for each $e\in \pi_1(H)$. We choose $0\neq \lambda_H\in X_*(T_H)_+$ such that $e_{\lambda}=0$. Then the map $\wedge^*\overline{H^{-1}(\frd_{\lambda,\s})}\cdot \eta_{e,\s}$ in Corollary \ref{cor:AB} induces an isomorphism of vector spaces \[\wedge^*(H^1(C,\frhc_{\s_H}^*)(1))\cong H^{-*}(\int_{X,e}\cF_{\s}).\] Therefore, the vector space $H^*(\int_{X,e}\cF_{\s})$ is identified with the free exterior algebra generated by the finite-dimensional vector space $H^1(C,\frhc_{\s_H}^*)(1)$, and the endomorphism algebra $H^*\mathbb{O}_{X,e,\s}=\End(H^*(\int_{X,e}\cF_{\s}))$ can be identified with the ring of differential operators on the odd vector space $H^1(C,\frhc_{\s_H}^*)(1)^*$. In particular, it admits a natural increasing filtration by the order of differential operators. We use $\{F_{\bullet}^{\Ran}H^*\mathbb{O}_{X,e,\s}\}_{\bullet\in \ZZ}$ to denote this filtration.

Under the canonical isomorphism $H^*\mathbb{O}_{X,\s}^{\red}\cong H^*\mathbb{O}_{X,e,\s}$, we obtain the Ran filtration $\{F_{\bullet}^{\Ran}H^*\mathbb{O}_{X,\s}^{\red}\}_{\bullet\in \ZZ}$. By Proposition \ref{prop:reducedLobs}, this filtration is independent of the choice of $e\in \pi_1(H)$.

This filtration is multiplicative and satisfies \begin{equation}\label{eq:Ranfilcommutator}
    [F_i^{\Ran}H^*\mathbb{O}_{X,\s}^{\red}, F_j^{\Ran}H^*\mathbb{O}_{X,\s}^{\red}] \sub F_{i+j-1}^{\Ran}H^*\mathbb{O}_{X,\s}^{\red}
\end{equation} for any $i,j\in \ZZ$.

\subsubsection{Associated graded algebra}
Since the filtration $F_{\bullet}^{\Ran}H^*\mathbb{O}_{X,\s}^{\red}$ on $H^*\mathbb{O}_{X,\s}^{\red}$ is multiplicative, there is an induced algebra structure on the associated graded vector space $\Gr^{\Ran}_{\bullet}H^*\mathbb{O}_{X,\s}^{\red}$, which admits the following description: 
\begin{prop}\label{prop:ABLobs}
    There is a canonical isomorphism of algebras \begin{equation} \wedge^{\bl}(H^1(C,\frhc_{\s_H}^*)(1)\oplus H^1(C,\frhc_{\s_H})) \isom \Gr_{\bullet}^{\Ran}H^*\mathbb{O}_{X,\s}^{\red} \end{equation} where $\bl$ is a grading variable taking integer values.
\end{prop} 
\begin{proof}
    This follows from Theorem \ref{thm:clifford}.
\end{proof}

\subsubsection{Characterizing the filtration via \texorpdfstring{$L$}{L}-observables}

\begin{prop}\label{prop:filtrationvialoc}
    The following holds: \begin{enumerate}
        \item For any $V\in \Rep(\Gc)^{\heartsuit}$ and any $\frc^{\loc}_V\in \Hom^0_{\Rep(\Gc)}(V, \PL_{X,\hbar})$, the induced map \[H^*(\glob(\frc^{\loc}_V)_{\s})^{\red}:H^{*+2}(C,V_{\s})(1)\to H^*\mathbb{O}_{X,\s}^{\red}\] defined in \eqref{eq:redobs} factors through $F_0^{\Ran}H^*\mathbb{O}_{X,\s}^{\red}$.

        \item For any $V'\in \Rep(\Gc)^{\heartsuit}$ and any $\frc^{\loc}_{V'}\in \Hom^0_{\Rep(\Gc)}(V', \PL_{X,\hbar}\langle 2\rangle)$, the induced map \[H^*(\glob(\frc^{\loc}_{V'})_{\s})^{\red}:H^*(C,V_{\s}')\to H^*\mathbb{O}_{X,\s}^{\red}\] factors through $F_1^{\Ran}H^*\mathbb{O}_{X,\s}^{\red}$.
    \end{enumerate}
\end{prop}

\begin{proof}
    If $g=1$, one has $\int_{X,e}\cF_{\s}\cong k$ for each $e\in\pi_1(H)$ and there is nothing to prove. We can assume $g\neq 1$. For simplicity, we assume that $|\pi_1(H)|=1$ and leave the general case to the reader.
    
    For (1), note that for any $V,W\in \Rep(\Gc)^{\heartsuit}$ and \[\frc_V^{\loc}\in \Hom^0_{\Rep(\Gc)}(V, \PL_{X,\hbar}),\quad \frc_W^{\loc}\in \Hom^0_{\Rep(\Gc)}(W, \PL_{X,\hbar}),\] since $[\frc_V^{\loc},\frc_W^{\loc}]=0$, by \cite[Theorem\,4.47]{liu2025higherperiodintegralsderivatives}, one has \begin{equation}\label{eq:commutator0}[\glob(\frc_V^{\loc}),\glob(\frc_W^{\loc})]=0.\end{equation} By Corollary \ref{cor:AB} and Proposition \ref{prop:diagh0}, we know $F_0^{\Ran}H^*\mathbb{O}_{X,\s}^{\red}$ is generated by diagonal $L$-observables in \eqref{eq:diagLobs}. Since elements in $\im(H^*(\glob(\frc_V^{\loc})_{\s})^{\red})$ commute with diagonal $L$-observables by \eqref{eq:commutator0}, we have $[\im(H^*(\glob(\frc_V^{\loc})_{\s})^{\red}),F_0H^*\mathbb{O}_{X,\s}^{\red}]=0$. One easily verifies that $F_0H^*\mathbb{O}_{X,\s}^{\red}\sub H^*\mathbb{O}_{X,\s}^{\red}$ is a maximal commutative subalgebra, which implies (1).

    For (2), note that a similar argument as above shows that $[\im(H^*(\glob(\frc_{V'}^{\loc})_{\s})^{\red}),F_0H^*\mathbb{O}_{X,\s}^{\red}]\sub F_0H^*\mathbb{O}_{X,\s}^{\red}$. One easily verifies that $F_1H^*\mathbb{O}_{X,\s}^{\red}\sub H^*\mathbb{O}_{X,\s}^{\red}$ is the maximal subspace such that $[F_1H^*\mathbb{O}_{X,\s}^{\red},F_0H^*\mathbb{O}_{X,\s}^{\red}]\sub F_0H^*\mathbb{O}_{X,\s}^{\red}$. This implies that $\im(H^*(\glob(\frc_{V'}^{\loc})_{\s}))^{\red}\sub F_1H^*\mathbb{O}_{X,\s}^{\red}$.
\end{proof}

Note that Proposition \ref{prop:filtrationvialoc} and Corollary \ref{cor:AB} together imply that the filtration $F_{\bullet}^{\Ran}H^*\mathbb{O}_{X,\s}^{\red}$ is independent of the choice of $\lambda_H\in X_*(T_H)_+$.

\subsection{Proof of Clifford relation}\label{sec:relproof}
In this section, we give the proof of Theorem \ref{thm:clifford2}.

Consider the cup product \begin{equation}
    \cup: \Gamma(C,\frhc_{\s_H})\otimes \Gamma(C,V_{\lambda,\s})\langle 2\rangle \to \Gamma(C, \frhc_{\s_H}\otimes V_{\lambda,\s})\langle 2\rangle\end{equation} and the map \begin{equation}
        \Gamma(\act_{V_{\lambda_H, \s_H}}\otimes\id): \Gamma(C,\frhc_{\s_H}\otimes V_{\lambda_H,\s_H}\otimes V_{\lambda_H,\s_H}^*)\langle 2\rangle\to \Gamma(C,V_{\lambda,\s})\langle 2\rangle
    .\end{equation}
We first observe the following primitive version of the Clifford relation.
\begin{lemma}\label{lem:cliffordprim}
    One has \[[\fra\frd_{\s}, \frd_{\lambda,\s}]=\frd_{\lambda,\s}\circ \Gamma(\act_{V_{\lambda_H, \s_H}}\otimes\id)\circ \cup \in \Hom^0(\Gamma(C,\frhc_{\s_H})\otimes \Gamma(C,V_{\lambda,\s})\langle 2\rangle, \mathbb{O}_{X,\s}).\]
\end{lemma}
\begin{proof}
    When $g=1$, the identity trivially holds since $\Gamma(C,\frhc_{\s_H})=0$. When $g\neq 1$, the identity follows from Proposition \ref{prop:globalcommutator}.
\end{proof}

\begin{proof}[Proof of Theorem \ref{thm:clifford2}]
The theorem is trivial when $\lambda=0$. We only consider the case when $e_{\lambda}=0$ and $\lambda\neq 0$ since the general case can be proved in a similar way.  Consider the subspace \[S_{\lambda} = \bigcap_{x\in H^2(C,V_{\lambda,\s})(1)} \ker(H^0(\frd_{\lambda,\s})(x)-H^2(\counit_{V_{\lambda_H},\s_H})(x))\sub H^*(\int_{X}\cF_{\s}) .\]

\begin{lemma}\label{lem:propertyofS}
    The following holds:
    \begin{enumerate}
        \item We have $\eta_{e,\s}\in S_{\lambda}$ for each $e\in \pi_1(H)$.
        \item $S_{\lambda}$ is stable under $\im (H^*(\frd_{\lambda, \s}))\sub H^*(\mathbb{O}_{X,\s})$.
    \end{enumerate}
\end{lemma}
\begin{proof}
    (1) follows from Proposition \ref{prop:diagh0} (the part we are using is independent of Theorem \ref{thm:clifford2}). (2) follows from the fact that diagonal $L$-observables are mutually commutative (see the proof of Proposition \ref{prop:filtrationvialoc}(1)).
\end{proof}

\begin{lemma}\label{lem:cliffordprim2}
    Theorem \ref{thm:clifford2} holds when restricted to the subspace $S_{\lambda}\sub H^*(\int_{X}\cF_{\s})$. That is, we have \[\begin{split}[H^1(\fra\frd_{\s}),H^{-1}(\frd_{\lambda,\s})]|_{S_{\lambda}}=\counit_{H^1(C,\frhc_{\s_H})}\circ (\id\otimes H^1(\act_{V_{\lambda_H},\s_H}'))|_{S_{\lambda}} \\ \in \Hom^0(H^1(C,\frhc_{\s_H})\otimes H^1(C,V_{\lambda,\s})(1), \Hom^0(S_{\lambda},H^*(\int_{X}\cF_{\s}))).\end{split}\]
\end{lemma}

\begin{proof}
    By Lemma \ref{lem:cliffordprim}, we only need to show \[\counit_{H^1(C,\frhc_{\s_H})}\circ (\id\otimes H^1(\act_{V_{\lambda_H},\s_H}'))|_{S_{\lambda}} = H^0(\frd_{\lambda,\s})\circ H^2(\act_{V_{\lambda_H, \s_H}}\otimes\id)\circ \cup|_{S_{\lambda}}.\] Since $H^0(\frd_{\lambda,\s})|_{S_{\lambda}} = H^2(\counit_{V_{\lambda_H},\s_H})|_{S_{\lambda}}$ by the definition of $S_{\lambda}$, we only need to show \[\counit_{H^1(C,\frhc_{\s_H})}\circ (\id\otimes H^1(\act_{V_{\lambda_H},\s_H}'))= H^2(\counit_{V_{\lambda_H},\s_H})\circ H^2(\act_{V_{\lambda_H, \s_H}}\otimes\id)\circ \cup .\] This follows from the commutative diagram 
    \[\begin{tikzcd}
        \frhc\otimes V_{\lambda_H}\otimes V_{\lambda_H}^* \arrow[r, "\act_{V_{\lambda_H}}\otimes\id"] \arrow[d, "\id\otimes\act_{V_{\lambda_H}}'"'] & V_{\lambda_H}\otimes V_{\lambda_H}^* \ar[d, "\counit_{V_{\lambda_H}}"]  \\
        \frhc\otimes \frhc^* \arrow[r, "\counit_{\frhc}"] &  k
    \end{tikzcd}.\]

\end{proof}

We show by induction that $H^{\geq -m}(\int_X\cF_{\s})\sub S_{\lambda}$ for $m\in \ZZ_{\geq 0}$. The base case $m=0$ is Lemma \ref{lem:propertyofS}(1). Assuming $H^{\geq -m+1}(\int_X\cF_{\s})\sub S_{\lambda}$ holds, we would like to show $H^{-m}(\int_X\cF_{\s})\sub S_{\lambda}$.

Consider a splitting $H^1(C,V_{\lambda,\s})(1)=H^1(C,\frhc_{\s_H}^*)(1)\oplus K$ for the quotient map \[H^1(\act_{V_{\lambda_H},\s_H}'):H^1(C,V_{\lambda,\s})(1)\to H^1(C,\frhc_{\s_H}^*)(1).\] By Lemma \ref{lem:cliffordprim2}, we have \[\begin{split}[H^1(\fra\frd_{\s}),H^{-1}(\frd_{\lambda,\s})|_{H^1(C,\frhc_{\s_H}^*)(1)}]|_{S_{\lambda}}=\counit_{H^1(C,\frhc_{\s_H})}|_{S_{\lambda}} \\ \in \Hom^0( H^1(C,\frhc_{\s_H})\otimes H^1(C,\frhc_{\s_H}^*)(1), \Hom^0(S_{\lambda},H^*(\int_{X}\cF_{\s})))\end{split}.\] From this, a standard argument (see the proof of \cite[Lemma\,6.16]{wang2025specialcycleshtukascategorical}) implies that the natural map \[\wedge^{\leq m}(H^1(C,\frhc_{\s_H}^*)(1))\to H^{\geq -m}(\int_{X,e}\cF_{\s})\] given by acting on the element $\eta_{e,\s}\in H^0(\int_{X,e}\cF_{\s})$ is injective for each $e\in \pi_1(H)$. By Proposition \ref{prop:dim}, this map is also surjective, which implies that $H^{-m}(\int_{X}\cF_{\s})\sub S_{\lambda}$ by Lemma \ref{lem:propertyofS}. This concludes the proof of Theorem \ref{thm:clifford2} when $e_{\lambda}=0$.
\end{proof}

\subsection{Intersection observable}\label{sec:intobs}
Now we introduce and study the intersection observable, which is a key ingredient for the proof of Theorem \ref{thm:main}.
\subsubsection{Definition of the intersection observable}
Consider the map \begin{equation}\rml: \mathbb{O}_{X,\s}\to \End(\mathbb{O}_{X,\s}),~\rml(x)(y)=x\circ y\end{equation}
\begin{equation}\rmr:\mathbb{O}_{X,\s}\to \End(\mathbb{O}_{X,\s}),~\rmr(x)(y) = (-1)^{|x||y|}y\circ x\end{equation} given by left and right multiplication. Define the adjoint map \begin{equation}
    \rmad: \mathbb{O}_{X,\s}\to \End(\mathbb{O}_{X,\s}),~\rmad(x)(y)=[x,y]=x\circ y - (-1)^{|x||y|}y\circ x
.\end{equation}
Consider the multiplication map \begin{equation}
    \rmm: \End(\mathbb{O}_{X,\s})\otimes \End(\mathbb{O}_{X,\s})\to \End(\mathbb{O}_{X,\s}),~\rmm(x\otimes y)=x\circ y
.\end{equation} We will also freely use $\rml$, $\rmr$, $\rmad$, and $\rmm$ to denote the similarly defined maps for $H^*\mathbb{O}_{X,\s}^{\red}$ in place of those for $\mathbb{O}_{X,\s}$.

Consider the canonical element \begin{equation} \unit_{\Gamma(C,V_{-w_0\lambda,\s})} \in H^0(\Gamma(C,V_{\lambda,\s})\otimes\Gamma(C, V_{-w_0\lambda,\s})\langle 2\rangle) \end{equation} induced by the element $\unit_{V_{-w_0\lambda}}\in \Hom^0_{\Rep(\Gc)}(k, V_{\lambda}\otimes V_{-w_0\lambda})$ given by the Satake cycle.

We are interested in the \emph{intersection observable} \begin{equation}
    \Gamma_{\lambda,\s}:= \rmm\circ (\rml\otimes\rmr)(\frd_{\det,\lambda,\s}\otimes \frd_{-w_0\lambda,\s}) (\unit_{\Gamma(C,V_{-w_0\lambda,\s})})\in \End(\mathbb{O}_{X,\s}).
\end{equation} 
By Proposition \ref{prop:reducedLobs}, we have the \emph{reduced intersection observable} \begin{equation}
    H^*\Gamma_{\lambda,\s}^{\red}=\rmm\circ (\rml\otimes\rmr)(H^*(\frd_{\det,\lambda,\s})^{\red}\otimes H^*(\frd_{-w_0\lambda,\s})^{\red}) (\unit_{\Gamma(C,V_{-w_0\lambda,\s})}) \in \End(H^*\mathbb{O}_{X,\s}^{\red})
\end{equation} such that under the isomorphism \eqref{eq:transiso}, we have \[H^*\Gamma_{\lambda,\s}=H^*\Gamma_{\lambda,\s}^{\red}\otimes \Ad_{T_{e_{\lambda}}} \] where \begin{equation}\label{eq:adt} \operatorname{Ad}_{T_{e_\lambda}}
\in
\operatorname{End}(\operatorname{End}(k[\pi_1(H)])),
\qquad
M\longmapsto T_{e_\lambda}MT_{-e_\lambda}.\end{equation}

\subsubsection{Decomposition}
We define the \emph{derivative part} of the (reduced) intersection observable to be \begin{equation}
    \nabla_{\lambda,\s}:= -\rmm\circ (\rml\otimes\rmad)(H^*(\frd_{\det,\lambda,\s})^{\red}\otimes H^*(\frd_{-w_0\lambda,\s})^{\red}) (\unit_{\Gamma(C,V_{-w_0\lambda,\s})})\in \End(H^*\mathbb{O}_{X,\s}^{\red}).
\end{equation} and the \emph{scalar part} of the (reduced) intersection observable to be \begin{equation}
    B_{\lambda,\s}:= \rmm\circ (\rml\otimes\rml)(H^*(\frd_{\det,\lambda,\s})^{\red}\otimes H^*(\frd_{-w_0\lambda,\s})^{\red}) (\unit_{\Gamma(C,V_{-w_0\lambda,\s})})\in \End(H^*\mathbb{O}_{X,\s}^{\red}).
\end{equation} It is clear that $H^*\Gamma_{\lambda,\s}^{\red}=\nabla_{\lambda,\s}+B_{\lambda,\s}$.

The scalar part $B_{\lambda,\s}$ is described by the following proposition.
\begin{prop}\label{prop:scalarpart}
    We have $B_{\lambda,\s}=(-1)^{d_{\lambda_H}}(2g-2)b_{\lambda}\in k\sub \End(H^*\mathbb{O}_{X,\s}^{\red})$ where $b_{\lambda}$ is given in \eqref{eq:bformula}.
\end{prop}

\begin{proof}
    Consider $\unit_{V_{-w_0\lambda,\s}}\in H^0(C, V_{\lambda,\s}\otimes V_{-w_0\lambda,\s})$ and $1 \in H^0(C,\uk_{C})\cong k$. We compute
    \[\begin{split}
        & \rmm\circ (\rml\otimes\rml)(\frd_{\det,\lambda,\s}\otimes \frd_{-w_0\lambda,\s}) (\unit_{\Gamma(C,V_{-w_0\lambda,\s})}) \\
        = & \rml(\rmm\circ (\frd_{\det,\lambda,\s}\otimes \frd_{-w_0\lambda,\s})(\unit_{\Gamma(C,V_{-w_0\lambda,\s})})) \\ 
        = & \rml(\rmm\circ (\frd_{\det,\lambda,\s}\otimes \frd_{-w_0\lambda,\s})\circ \D_{C,!}(\unit_{V_{-w_0\lambda,\s}})) \\
        = & \rml((\frd_{\det,\lambda}\circ \frd_{-w_0\lambda})_{\s}(\unit_{V_{-w_0\lambda,\s}})) \\ 
        = & \rml(\glob(\unit_{V_{-w_0\lambda}}^*(\frd_{\det,\lambda}^{\loc}\circ \frd_{-w_0\lambda}^{\loc}) )_{\s}(1) ) \\
        = & \rml(\glob((-1)^{d_{\lambda_H} +1}b_{\lambda}\hbar)_{\s}(1)) \\
        = & \rml((-1)^{d_{\lambda_H} +1}(2-2g)b_{\lambda}) \\
        = & (-1)^{d_{\lambda_H}}(2g-2)b_{\lambda}.
    \end{split}\]
Here, the third equality follows from \eqref{eq:Lobsassocfusion}, the fourth equality follows from \eqref{eq:ltogassoc}, the fifth equality follows from Proposition \ref{prop:scalarpartloc}, the sixth equality follows from the fact that the natural coordinate map $\mathrm{Coor}: C\to \mathbb{B}\Aut(D)$ satisfies $\mathrm{Coor}^*(\hbar)=c_1(T_C)\in H^2(C,\uk_C)$ where $T_C$ is the tangent bundle of $C$, and the remaining equalities are straightforward.
\end{proof}

For the derivative part $\nabla_{\lambda,\s}$, we first note the following.
\begin{lemma}\label{lem:derivativepartfil}
    The operator $\nabla_{\lambda,\s}\in \End(H^*\mathbb{O}_{X,\s}^{\red})$ preserves the Ran filtration $F_{\bullet}^{\Ran}H^*\mathbb{O}_{X,\s}^{\red}$ defined in \S\ref{sec:deffil}. Moreover, the induced operator on the associated graded algebra \[\Gr^{\Ran}_{\bullet}\nabla_{\lambda,\s}\in \End(\Gr^{\Ran}_{\bullet}H^*\mathbb{O}_{X,\s}^{\red})\] is a $k$-linear derivation.
\end{lemma}

\begin{proof}
    By Proposition \ref{prop:filtrationvialoc}, we have \[\im(H^*(\frd_{\det,\lambda,\s})^{\red})\sub F_1^{\Ran}H^*\mathbb{O}_{X,\s}^{\red}\] and \[\im(H^*(\frd_{-w_0\lambda,\s})^{\red})\sub F_0^{\Ran}H^*\mathbb{O}_{X,\s}^{\red}.\] By \eqref{eq:Ranfilcommutator}, we know \[\rml(\im(H^*(\frd_{\det,\lambda,\s})^{\red}))(F_i^{\Ran}H^*\mathbb{O}_{X,\s}^{\red})\sub F_{i+1}^{\Ran}H^*\mathbb{O}_{X,\s}^{\red}\] and \[\rmad(\im(H^*(\frd_{-w_0\lambda,\s})^{\red}))(F_i^{\Ran}H^*\mathbb{O}_{X,\s}^{\red})\sub F_{i-1}^{\Ran}H^*\mathbb{O}_{X,\s}^{\red}\] for any $i\in \ZZ$. This proves both claims.
\end{proof}

Then we have the following description of $\Gr^{\Ran}_{\bullet}\nabla_{\lambda,\s}$.
\begin{prop}\label{prop:derivativepart}
    We have \[\Gr^{\Ran}_{\bullet}\nabla_{\lambda,\s}=\nabla_{E_{\lambda,\s}} \in \End(\wedge^{\bl}(H^1(C,\frhc_{\s_H}^*)(1)\oplus H^1(C,\frhc_{\s_H})))\] where we are using the isomorphism $\wedge^{\bl}(H^1(C,\frhc_{\s_H}^*)(1)\oplus H^1(C,\frhc_{\s_H}))\isom \Gr^{\Ran}_{\bullet}H^*\mathbb{O}_{X,\s}^{\red}$ in Proposition \ref{prop:ABLobs}, and $\nabla_{E_{\lambda,\s}}$ is the $k$-linear derivation on $\wedge^{\bl}(H^1(C,\frhc_{\s_H}^*)(1)\oplus H^1(C,\frhc_{\s_H}))$ induced by \[E_{\lambda,\s}=(0, (-1)^{d_{\lambda_H} + 1}\epsilon_{\lambda}\id)\in \End(H^1(C,\frhc_{\s_H}^*)(1)\oplus H^1(C,\frhc_{\s_H}))\] where $\epsilon_{\lambda}$ is given in \eqref{eq:epsilonformula}.
\end{prop}

\begin{proof}
    By Proposition \ref{prop:reducedLobs}, we can assume that $|\pi_1(H)|=1$ without loss of generality. We only need to show that $\Gr^{\Ran}_{\bullet}\nabla_{\lambda,\s}$ and $\nabla_{E_{\lambda,\s}}$ coincide on the generators in $H^1(C,\frhc_{\s_H}^*)(1)$ and $H^1(C,\frhc_{\s_H})$.

    For the first part, note that $\im(H^*(\frd_{-w_0\lambda,\s}))\sub F_0^{\Ran}H^*\mathbb{O}_{X,\s}$, which implies that $\nabla_{\lambda,\s}|_{F_0^{\Ran}H^*\mathbb{O}_{X,e,\s}}=0$. Therefore, $\Gr^{\Ran}_{\bullet}\nabla_{\lambda,\s}|_{H^1(C,\frhc_{\s_H}^*)(1)}=0=\nabla_{E_{\lambda,\s}}|_{H^1(C,\frhc_{\s_H}^*)(1)}$. 

    For the second part, note that  \[\begin{split}
    \nabla_{\lambda,\s}(\fra\frd_{\s}(-)) & = -\rmm\circ (\rml\otimes\rmad)(\frd_{\det,\lambda,\s}\otimes \frd_{-w_0\lambda,\s}) (\unit_{\Gamma(C,V_{-w_0\lambda,\s})})(\fra\frd_{\s}(-)) \\
    & = - \rmm\circ (\frd_{\det,\lambda,\s}\otimes [\frd_{-w_0\lambda,\s},\fra\frd_{\s}]) (\unit_{\Gamma(C,V_{-w_0\lambda,\s})}\otimes -) \\
    & = - \rmm\circ (\frd_{\det,\lambda,\s}\otimes [\frd_{-w_0\lambda},\fra\frd]_{\s}) (\unit_{\Gamma(C,V_{-w_0\lambda,\s})}\cup -) \\
    & = \rmm\circ (\frd_{\det,\lambda,\s}\otimes ((\act_{V_{-w_0\lambda_H}}\boxtimes \id)^*\frd_{-w_0\lambda})_{\s}) (\unit_{\Gamma(C,V_{-w_0\lambda,\s})}\cup -) \\
    & =  \rmm\circ (\frd_{\det,\lambda,\s}\otimes ((\act_{V_{-w_0\lambda_H}}\boxtimes \id)^*\frd_{-w_0\lambda})_{\s})\circ \D_{C,!}(\unit_{V_{-w_0\lambda,\s}}\cup -) \\
    & =  (\frd_{\det,\lambda}\circ (\act_{V_{-w_0\lambda_H}}\boxtimes \id)^*\frd_{-w_0\lambda})_{\s} (\unit_{V_{-w_0\lambda,\s}}\cup -) \\
    & =  \glob ((\id\otimes\unit_{V_{-w_0\lambda}})^*(\frd_{\det,\lambda}^{\loc}\circ (\act_{V_{-w_0\lambda_H}}\boxtimes \id)^*\frd_{-w_0\lambda}^{\loc}))_{\s} (-) \\
    & = \glob ((-1)^{d_{\lambda_H} + 1 }\epsilon_{\lambda}\fra\frd^{\loc})_{\s} (-) \\
    & = (-1)^{d_{\lambda_H} + 1 }\epsilon_{\lambda}\fra\frd_{\s}(-)
    \end{split}.\]
    Here, the fourth equality follows from Proposition \ref{prop:globalcommutator}, the sixth equality follows from \eqref{eq:Lobsassocfusion}, the seventh equality follows from \eqref{eq:ltogassoc}, the eighth equality follows from Proposition \ref{prop:localunit}, and the remaining equalities are straightforward. This implies that $\Gr^{\Ran}_{\bullet}\nabla_{\lambda,\s}|_{H^1(C,\frhc_{\s_H})}=(-1)^{d_{\lambda_H} + 1 }\epsilon_{\lambda}\id=E_{\lambda,\s}|_{H^1(C,\frhc_{\s_H})}$, which concludes the proof of the proposition.
\end{proof}

Combining Proposition \ref{prop:scalarpart} and Lemma \ref{lem:derivativepartfil}, we know that $H^*\Gamma_{\lambda,\s}$ preserves the Ran filtration \[F_{\bullet}^{\Ran}H^*\mathbb{O}_{X,\s}=(F_{\bullet}^{\Ran}H^*\mathbb{O}_{X,\s}^{\red})\otimes \End(k[\pi_1(H)]),\] and we have the following description of the associated graded operator $\Gr^{\Ran}_{\bullet}H^*\Gamma_{\lambda,\s}$:

\begin{prop}\label{prop:gammagraded}
    We have \[\begin{split}
        \Gr^{\Ran}_{\bullet}H^*\Gamma_{\lambda,\s}=(\nabla_{E_{\lambda,\s}} + (-1)^{d_{\lambda_H}}(2g-2)b_{\lambda}\id)\otimes \Ad_{T_{e_{\lambda}}}\\
        \in \End(\wedge^{\bl}(H^1(C,\frhc_{\s_H}^*)(1)\oplus H^1(C,\frhc_{\s_H}))\otimes \End(k[\pi_1(H)]))
    \end{split}\] where $T_{e_{\lambda}}$ is defined in \eqref{eq:shiftobs}.

\end{prop}

The proof is straightforward.

\section{Diagonal cycle and intersection number}\label{sec:diagcycle}
In this section, we work towards the proof of Theorem \ref{thm:main} and Theorem \ref{thm:intnondegintro}. 
\subsection{Cohomology of Shtukas}\label{sec:cohshtuka}
We first recall some background on the cohomology of Shtukas.

We use $\Loc_{\Gc}^{\res}$ to denote the moduli stack of $\Gc$-local systems with restricted variation, and $\Loc_{\Gc}^{\arith}$ to denote the moduli stack of Weil $\Gc$-local systems as considered in \cite{arinkin2022stacklocalsystemsrestricted}. There is a forgetful map $\Loc_{\Gc}^{\arith}\to \Loc_{\Gc}^{\res}$.

Consider the map $l_I:\Sht_{G,I}\to C^I$ remembering the legs. For each $\lambda_I\in X_*(T)_+^I$, we have the intersection cohomology sheaf $\IC_{\lambda_I}\in \Shv(\Sht_{G,I})$ normalized to be perverse along fibers of $l_I$. Throughout this section, we assume that $\Sht_{G,\leq \lambda_I}\neq \varnothing$.

The cohomology of Shtukas is defined to be $l_{I,!} \IC_{\lambda_I}\in \Shv(C^I)$, or sometimes, we mean its global section $\Gamma_c(\Sht_{G,I}, \IC_{\lambda_I})$. Both of them carry an action by the algebra of excursion operators $\cO(\Loc_{\Gc}^{\arith})$.

Let $\s\in \Loc_{\Gc}^{\arith}(k)$ be a geometrically strongly irreducible Weil $\Gc$-local system as defined in Definition \ref{defn:stronglyirreducible}. In this case, the underlying reduced substack of $\Loc_{\Gc}^{\res}$ containing $\s$ is isomorphic to $\mathbb{B}Z(\Gc)$. Also, the connected component of $\Loc_{\Gc}^{\arith}$ containing $\s$ (we denote it by $(\Loc_{\Gc}^{\arith})_{\s}$) is isomorphic to $\mathbb{B}Z(\Gc)$.

We use $(l_{I,!}\IC_{\lambda_I})_{\s}\in \Shv(C^I)$ (resp. $\Gamma_c(\Sht_{G,I}, \IC_{\lambda_I})_{\s}$) to denote the direct summand of $l_{I,!}\IC_{\lambda_I}$ (resp. $\Gamma_c(\Sht_{G,I}, \IC_{\lambda_I})$) on which $\cO(\Loc_{\Gc}^{\arith})$ acts through $\cO((\Loc_{\Gc}^{\arith})_{\s})$. This is the $\s$-isotypic component of the cohomology of shtukas, which admits the following description:

\begin{prop}\label{prop:cohshtuka}
If $\s$ is automorphic, there is an isomorphism $\xi_{\s,I}: V_{\lambda_I,\s}\isom (l_{I,!}\IC_{\lambda_I})_{\s}$.
\end{prop}

\begin{proof}
    The map $\xi_{\s,I}$ is constructed via a Weil Hecke eigensheaf $\cF_{\s}\in \Shv(\Bun_G)$ with eigenvalue $\s$ as in \cite[(5.6)]{wang2025specialcycleshtukascategorical}, whose existence is guaranteed by \cite[Theorem\,0.1.4]{gaitsgory2025geometriclanglandspositivecharacteristic} and \cite[Main Theorem\,0.2.6]{arinkin2022automorphicfunctionstracefrobenius} under the assumption that $\s$ is automorphic. It is straightforward to verify that the image of $\xi_{\s,I}$ is contained in the $\s$-isotypic component $(l_{I,!}\IC_{\lambda_I})_{\s}$. By Corollary \ref{cor:intnondegenerate}, the map $\xi_{\s,I}$ is injective. By \cite[Main Theorem\,0.3.10]{arinkin2022automorphicfunctionstracefrobenius}, we know that the source and the target are abstractly isomorphic, which implies that $\xi_{\s,I}$ is an isomorphism.
\end{proof}

Note that the isomorphism $\xi_{\s,I}$ depends on the choice of $\cF_{\s}$, which we will always fix once and for all. Moreover, for the $\Gc$-local system $c_{\Gc}(\s)$ obtained by applying the Cartan involution to $\s$, we will fix $\cF_{c_{\Gc}(\s)}=\mathbb{D}^{\mathrm{Ver}}(\cF_{\s})$.

\subsection{Diagonal cycle and intersection pairing}
Now we specialize to the group case and study the diagonal cycle and intersection pairing. 

We consider the case $G=H\times H$, and $\s=(\s_H,c_{\Hc}(\s_H))$ for some automorphic $\s_H\in \Loc_{\Hc}^{\arith}(k)$. We choose a Weil Hecke eigensheaf $\cF_{\s}$ as in \eqref{eq:eigensheaf}.

For $\lambda_{H,I}\in X_*(T_H)_+^I$ and $\lambda_I=(\lambda_{H,I},\lambda_{H,I})\in X_*(T)_+^I$, consider the diagonal map $\D_{\Sht}:\Sht_{H,I}\to \Sht_{G,I}$. Consider the fundamental class \begin{equation}
    [\Sht_{H,\leq\lambda_{H,I}}/C^I]\in \Hom^0(l_{H,I,!}\D_{\Sht}^*\IC_{\lambda_I}, \uk_{C^I})\cong \Hom^0(l_{H,I,!}\uk_{\Sht_{H,\leq\lambda_{H,I}}}\langle 2 d_{\lambda_{H,I}}\rangle, \uk_{C^I} )
\end{equation} where $l_{H,I}:\Sht_{H,I}\to C^I$ is the map remembering the legs. Define the diagonal cycle class to be \begin{equation}
    \D_{\Sht,!}[\Sht_{H,\leq\lambda_{H,I}}/C^I]\in \Hom^0(l_{I,!}\IC_{\lambda_I}, \uk_{C^I})\cong IH^{d_{\lambda_I}+2r}_c(\Sht_{G,\leq\lambda_I})(d_{\lambda_{H,I}}+r)^*
.\end{equation}

Consider the diagonal map $\D_{\Sht,2}:\Sht_{G,I}\to\Sht_{G\times G,I}$. The same construction as above gives us the diagonal cycle class (i.e. intersection pairing) \begin{equation}\label{eq:intpairing}
    \langle~,~\rangle_{\lambda_I}:=\D_{\Sht,2,!}[\Sht_{G,\leq\lambda_I}/C^I]\in \Hom^0((l_{I,!}\IC_{\lambda_I})^{\otimes 2}, \uk_{C^I}).
\end{equation} It induces the intersection pairing \eqref{eq:intpairingintro} by taking cohomology and cup product.

\subsection{Intersection number as a trace}

We would like to study the $\s$-isotypic part of various cycle classes defined above: \begin{equation}\begin{split}
    (\D_{\Sht,!}[\Sht_{H,\leq\lambda_{H,I}}/C^I])_{\s}:=(\D_{\Sht,!}[\Sht_{H,\leq\lambda_{H,I}}/C^I])|_{\xi_{\s,I}} \\ \in \Hom^0(V_{\lambda_I,\s}, k) \cong H^{2r}(C^I,V_{\lambda_I,\s})(r)^*
\end{split}\end{equation}
\begin{equation}\begin{split}
    \Bigl(\D_{\Sht,!}\bigl([\Sht_{H,\leq\lambda_{H,I}}/C^I]\cup \prod_{i=1}^r c_1(\cL_{\det,i})\bigr)\Bigr)_{\s}:=\Bigl(\D_{\Sht,!}\bigl([\Sht_{H,\leq\lambda_{H,I}}/C^I]\cup \prod_{i=1}^r c_1(\cL_{\det,i})\bigr)\Bigr)\Big|_{\xi_{\s,I}} \\ \in \Hom^0(V_{\lambda_I,\s}, k\langle 2r\rangle) \cong H^0(C^I,V_{\lambda_I,\s})^*
\end{split}\end{equation}
\begin{equation}
    \langle~,~\rangle_{\lambda_I,\s}:=\langle~,~\rangle_{\lambda_I}|_{\xi_{\s,I}\otimes \xi_{c_{\Gc}(\s),I}}\in \Hom^0(V_{\lambda_I,\s}\otimes V_{\lambda_I,c_{\Gc}(\s)}, k)\cong \Hom^0(V_{\lambda_I,\s}\otimes V_{-w_0\lambda_I,\s}, k)
\end{equation} where $c_{\Gc}:\Gc\to \Gc$ is the Cartan involution. To do this, consider \begin{equation}
    z_{\lambda_I,\s} := \tr(\Frob\circ \frd_{\lambda_I,\s}, \int_X\cF_{\s})\in H^{2r}(C^I,V_{\lambda_I,\s})(r)^*
\end{equation}
\begin{equation}
    z_{\det,\lambda_I,\s} := \tr(\Frob\circ \frd_{\det,\lambda_I,\s}, \int_X\cF_{\s} )\in H^0(C^I, V_{\lambda_I,\s})^*.
\end{equation} Also, there is a canonical counit \begin{equation}
    \counit_{V_{\lambda_I,\s}}:V_{\lambda_I,\s}\otimes V_{-w_0\lambda_I,\s}\to \uk_{C^I}
\end{equation} induced by $\counit_{V_{\lambda_I}}:V_{\lambda_I}\otimes V_{-w_0\lambda_I}\to k$ given by the Satake cycle.

\begin{prop}\label{prop:cycle=trace}
    The following holds:
    \begin{enumerate}
        \item $(\D_{\Sht,!}[\Sht_{H,\leq\lambda_{H,I}}/C^I])_{\s} = z_{\lambda_I,\s}\in H^{2r}(C^I,V_{\lambda_I,\s})(r)^*$.
        \item $\Bigl(\D_{\Sht,!}\bigl([\Sht_{H,\leq\lambda_{H,I}}/C^I]\cup \prod_{i=1}^r c_1(\cL_{\det,i})\bigr)\Bigr)_{\s} = z_{\det,\lambda_I,\s}\in H^0(C^I, V_{\lambda_I,\s})^*$.
        \item $\langle~,~\rangle_{\lambda_I,\s}=\tr(\Frob, \Gamma_c(\Bun_G,\cF_{\s}\otimes \cF_{c_{\Gc}(\s)}))\cdot \counit_{V_{\lambda_I,\s}}\in \Hom^0(V_{\lambda_I,\s}\otimes V_{-w_0\lambda_I,\s}, \uk_{C^I}).$
    \end{enumerate}
\end{prop}

\begin{proof}
    For (1) and (3), when $H=\GL_n$, they follow from \cite[Proposition\,6.11]{wang2025specialcycleshtukascategorical}. The proof in \textit{loc. cit.} relies on \cite[Assumption\,5.13]{wang2025specialcycleshtukascategorical}, which is only verified when $H=\GL_n$ in \textit{loc. cit.} but now confirmed in general by Corollary \ref{cor:AB}. Therefore, the same argument as in \textit{loc. cit.} works in general.

    For (2), one can slightly modify the proof in \textit{loc. cit.} for (1). We adopt the notations in \textit{loc. cit.} One replaces the identity in \cite[Proposition\,6.11, Theorem\,3.3]{wang2025specialcycleshtukascategorical} by \begin{equation}
    \D_{\Sht,!}\bigl([\Sht_{H,\leq\lambda_{H,I}}/C^I]\cup \prod_{i=1}^r c_1(\cL_{\det,i})\bigr)=\tr_{\Sht,C^I}(\frd_{\det,\lambda_I}).
\end{equation} Then the same argument as in \textit{loc. cit.} works. To prove this identity, one only needs to prove \begin{equation}
[\Sht_{H,\leq\lambda_{H,I}}/C^I]\cup \prod_{i=1}^r c_1(\cL_{\det,i})=\tr_{\Sht,C^I}([\Hk_{H,\leq\lambda_{H,I}}/\Bun_H\times C^I]\cup \prod_{i=1}^r c_1(\cL_{\det,i})).
\end{equation} This follows from the identity \begin{equation}
[\Sht_{H,\leq\lambda_{H,I}}/C^I]=\tr_{\Sht,C^I}([\Hk_{H,\leq\lambda_{H,I}}/\Bun_H\times C^I])
\end{equation} in the proof of \cite[Theorem\,3.3]{wang2025specialcycleshtukascategorical} and the fact that $\tr_{\Sht,C^I}$ is linear over $\Gamma(\Hk_{H,I},\uk_{\Hk_{H,I}})$.

\end{proof}

The following corollary is a refinement of Theorem \ref{thm:intnondegintro}.
\begin{cor}\label{cor:intnondegenerate}
    One has \begin{equation}
        \langle~,~\rangle_{\lambda_I,\s}=|\pi_1(H)|^2L(1, \frhc_{\s_H})^2\cdot  \counit_{V_{\lambda_I,\s}}\in \Hom^0(V_{\lambda_I,\s}\otimes V_{-w_0\lambda_I,\s}, \uk_{C^I})
    .\end{equation} In particular, the intersection pairing $\langle~,~\rangle_{\lambda_I,\s}\in \Hom^0(V_{\lambda_I,\s}\otimes V_{-w_0\lambda_I,\s}, \uk_{C^I})$ is nondegenerate.
\end{cor}

\begin{proof}
    The proof of the identity is identical to \cite[\S6.3.2]{wang2025specialcycleshtukascategorical}. By \cite[Proposition\,25.4.6]{arinkin2022stacklocalsystemsrestricted}, we know that the Weil local system $\frhc_{\s_H}$ is pure of weight $0$. Therefore, we have $L(1,\frhc_{\s_H})\neq 0$, which implies the nondegeneracy of the intersection pairing.
\end{proof}

\subsection{Relation with the intersection observable}
Now we relate the desired intersection number in Theorem \ref{thm:main} to the trace of the intersection observable introduced in \S\ref{sec:intobs}.

For $\lambda_{H,I}=(\lambda_{H,1},\cdots,\lambda_{H,r})\in X_*(T_H)_+^I$ and $\lambda_I=(\lambda_{H,I},\lambda_{H,I})\in X_*(T)_+^I$, consider the multi-leg intersection observable \begin{equation}
    \Gamma_{\lambda_I,\s}:= \Gamma_{\lambda_1,\s}\circ \cdots\circ \Gamma_{\lambda_r,\s}\in \End(\mathbb{O}_{X,\s}) .\end{equation} 

\begin{lemma}\label{lem:intnumber=trace}
    The following identity holds: \begin{equation}\begin{split}
        \Bigl\langle \Bigl(\D_{\Sht,!}\bigl([\Sht_{H,\leq\lambda_{H,I}}/C^I]\cup \prod_{i=1}^r c_1(\cL_{\det,i})\bigr)\Bigr)_{\s} , (\D_{\Sht,!}[\Sht_{H,\leq\lambda_{H,I}}/C^I])_{c_{\Gc}(\s)} \Bigr\rangle^*_{\lambda_I,\s} \\ =  \frac{q^{-(g-1)\dim H}}{|\pi_1(H)|^2L(1, \frhc_{\s_H})^2} \tr(\Frob\circ H^*\Gamma_{\lambda_I,\s}, H^*\mathbb{O}_{X,\s}) \end{split} 
    \end{equation} where the pairing $\langle-,-\rangle^*_{\lambda_I,\s}$ is the dual pairing of $\langle-,-\rangle_{\lambda_I,\s}$ as in \eqref{eq:intpairingisotypicintrodual}.
\end{lemma}

\begin{proof}

Consider the unit
\[
    \unit_{V_{-w_0\lambda_I}}:k\to V_{\lambda_I}\otimes V_{-w_0\lambda_I}
\]
given by the Satake cycle. It induces an element
\[
    \unit_{V_{-w_0\lambda_I,\s}}
    \in H^0(C^I,V_{\lambda_I,\s}\otimes V_{-w_0\lambda_I,\s}),
\]
and hence an element
\[
    \unit_{\Gamma(C^I,V_{-w_0\lambda_I,\s})}
    \in H^0(\Gamma(C^I,V_{\lambda_I,\s})\otimes
    \Gamma(C^I,V_{-w_0\lambda_I,\s})\langle 2r\rangle).
\]
Since
\[
H^0(\Gamma(C^I,V_{\lambda_I,\s})\otimes \Gamma(C^I,V_{-w_0\lambda_I,\s})\langle 2r\rangle)
\cong
\bigoplus_{i=0}^{2r} H^i(C^I,V_{\lambda_I,\s})\otimes H^{2r-i}(C^I,V_{-w_0\lambda_I,\s})(r),
\]
we have a decomposition
\[
    \unit_{\Gamma(C^I,V_{-w_0\lambda_I,\s})}
    =\sum_{i=0}^{2r} \unit_{\Gamma(C^I,V_{-w_0\lambda_I,\s})}^i,
\]
where
\[
    \unit_{\Gamma(C^I,V_{-w_0\lambda_I,\s})}^i
    \in H^i(C^I,V_{\lambda_I,\s})\otimes H^{2r-i}(C^I,V_{-w_0\lambda_I,\s})(r).
\]
Consider also the counits \[\counit_{V_{\lambda_{H,I}}}:V_{\lambda_{H,I}}\otimes V_{-w_0\lambda_{H,I}}\to k\] \[\counit_{V_{-w_0\lambda_{H,I}}}:V_{-w_0\lambda_{H,I}}\otimes V_{\lambda_{H,I}}\to k.\]
They induce  \begin{equation}\label{eq:counitcartan} \counit_{V_{\lambda_{H,I},c_{\Hc}(\s_H)}} = \counit_{V_{-w_0\lambda_{H,I},\s_H}} :H^{2r}(C^I, V_{\lambda_I,c_{\Gc}(\s)})(r)\cong H^{2r}(C^I, V_{-w_0\lambda_I,\s})(r) \to k.\end{equation}

By Corollary \ref{cor:intnondegenerate}, we have \begin{equation}\label{eq:int*formula}\langle-,-\rangle_{\lambda_I,\s}^* = \frac{1}{|\pi_1(H)|^2L(1,\frhc_{\s_H})^2}\unit_{\Gamma(C^I,V_{-w_0\lambda_I,\s})}\in H^0(\Gamma(C^I,V_{\lambda_I,\s})\otimes \Gamma(C^I,V_{-w_0\lambda_I,\s})\langle 2r\rangle).\end{equation} We have \[\adjustbox{max width=\linewidth}{%
$%
\begin{aligned}
    & \Bigl\langle \Bigl(\D_{\Sht,!}\bigl([\Sht_{H,\leq\lambda_{H,I}}/C^I]\cup \prod_{i=1}^r c_1(\cL_{\det,i})\bigr)\Bigr)_{\s} , (\D_{\Sht,!}[\Sht_{H,\leq\lambda_{H,I}}/C^I])_{c_{\Gc}(\s)} \Bigr\rangle^*_{\lambda_I,\s} \\
     =& \frac{1}{|\pi_1(H)|^2L(1, \frhc_{\s_H})^2} (z_{\det,\lambda_I,\s}\otimes z_{\lambda_I,c_{\Gc}(\s)})(\unit_{\Gamma(C^I,V_{-w_0\lambda_I,\s})}^0) \\ 
     =& \frac{1}{|\pi_1(H)|^2L(1, \frhc_{\s_H})^2} \tr(\Frob\circ (H^0\frd_{\det,\lambda_I,\s}\otimes H^0\frd_{\lambda_I,c_{\Gc}(\s)})(\unit_{\Gamma(C^I,V_{-w_0\lambda_I,\s})}^0), H^*(\int_X\cF_{\s})\otimes H^*( \int_X\cF_{c_{\Gc}(\s)})) \\ 
     =& \frac{1}{|\pi_1(H)|^2L(1, \frhc_{\s_H})^2} \tr(\Frob\circ (H^0\frd_{\det,\lambda_I,\s}\otimes \counit_{V_{\lambda_{H,I},c_{\Hc}(\s_H)}})(\unit_{\Gamma(C^I,V_{-w_0\lambda_I,\s})}^0), H^*(\int_X\cF_{\s})\otimes H^*( \int_X\cF_{c_{\Gc}(\s)})) \\ 
     =& \frac{q^{-(g-1)\dim H}}{|\pi_1(H)|^2L(1, \frhc_{\s_H})^2} \tr(\Frob\circ (H^0\frd_{\det,\lambda_I,\s}\otimes\counit_{V_{-w_0\lambda_{H,I},\s_H}})(\unit_{\Gamma(C^I,V_{-w_0\lambda_I,\s})}^0), H^*(\int_X\cF_{\s})\otimes H^*( \int_X\cF_{\s})^*) \\ 
     =& \frac{q^{-(g-1)\dim H}}{|\pi_1(H)|^2L(1, \frhc_{\s_H})^2} \tr(\Frob\circ (\rmm\circ(\rml\otimes\rmr) (H^0\frd_{\det,\lambda_I,\s}\otimes \counit_{V_{-w_0\lambda_{H,I},\s_H}})(\unit_{\Gamma(C^I,V_{-w_0\lambda_I,\s})}^0)), H^*\mathbb{O}_{X,\s}) \\ 
     =& \frac{q^{-(g-1)\dim H}}{|\pi_1(H)|^2L(1, \frhc_{\s_H})^2} \tr(\Frob\circ (\rmm\circ(\rml\otimes\rmr) (H^0\frd_{\det,\lambda_I,\s}\otimes H^0\frd_{-w_0\lambda_I,\s})(\unit_{\Gamma(C^I,V_{-w_0\lambda_I,\s})}^0)), H^*\mathbb{O}_{X,\s}) \\ 
     =& \frac{q^{-(g-1)\dim H}}{|\pi_1(H)|^2L(1, \frhc_{\s_H})^2} \tr(\Frob\circ (\rmm\circ(\rml\otimes\rmr) (H^*\frd_{\det,\lambda_I,\s}\otimes H^*\frd_{-w_0\lambda_I,\s})(\unit_{\Gamma(C^I,V_{-w_0\lambda_I,\s})})), H^*\mathbb{O}_{X,\s}) \\ 
     =& \frac{q^{-(g-1)\dim H}}{|\pi_1(H)|^2L(1, \frhc_{\s_H})^2} \tr(\Frob\circ H^*\Gamma_{\lambda_I,\s}, H^*\mathbb{O}_{X,\s}) 
\end{aligned}
$%
.}\] Here, the first equality follows from \eqref{eq:int*formula} and Proposition \ref{prop:cycle=trace}, the third and sixth equalities follow from Proposition \ref{prop:diagh0}, the fourth equality follows from Corollary \ref{cor:duality} and \eqref{eq:counitcartan}, and the seventh equality holds because the two operators involved have the same semisimplification for cohomological degree reasons. The remaining equalities are straightforward. This concludes the proof of Lemma \ref{lem:intnumber=trace}.    

\end{proof}

\subsection{Proof of the main result}
Now we are ready to prove Theorem \ref{thm:main}.
\begin{proof}[Proof of Theorem \ref{thm:main}]
We have \[\begin{split}
& \tr(\Frob\circ H^*\Gamma_{\lambda_I,\s}, H^*\mathbb{O}_{X,\s}) \\ 
=& (-1)^{d_{\lambda_{H,I}}}|\pi_1(H)|^2\tr(\Frob\circ  \prod_{i=1}^r (\nabla_{(0,-\epsilon_{\lambda_i}\id)}+(2g-2)\frac{\dim(\frhc)}{2}\epsilon_{\lambda_i}), \wedge^{\bl}(H^1(C,\frhc_{\s_H}^*)(1)\oplus H^1(C,\frhc_{\s_H}))) \\ 
=& |\pi_1(H)|^2 L(1, \frhc_{\s_H})\left(\prod_{i=1}^r(\epsilon_{\lambda_i}\frac{1}{\log q}\frac{d}{ds}+(2g-2)\frac{\dim(\frhc)}{2}\epsilon_{\lambda_i})\right)L(s, \frhc_{\s_H})\Big|_{s=0}
\end{split}
\] 
where the first equality follows from Proposition \ref{prop:gammagraded}, the second equality follows from a standard computation and $(-1)^{d_{\lambda_{H,I}}}=1$ as $\Sht_{H,\leq \lambda_{H,I}}\neq \varnothing$. 

Then the desired identity follows from combining the above identity with Lemma \ref{lem:intnumber=trace} and the fact $L(1,\frhc_{\s_H})= q^{-(g-1)\dim H}L(0,\frhc_{\s_H})$. This concludes the proof of Theorem \ref{thm:main}.

\end{proof}

\bibliographystyle{amsalpha}
\bibliography{Bibliography}
\end{document}